\documentclass[12pt]{article} %
\usepackage{amssymb,amsmath,amsthm,amsfonts}
\usepackage[vlined, ruled, linesnumbered]{algorithm2e}
\usepackage[T1]{fontenc}
\usepackage[utf8]{inputenc}
\usepackage{enumerate}
\usepackage{subfigure}

\newtheorem{theorem}{Theorem}
\newtheorem{cor}[theorem]{Corollary}
\newtheorem{lemma}[theorem]{Lemma}

\newtheorem{prop}[theorem]{Proposition}

\theoremstyle{definition}
\newtheorem{defi}[theorem]{Definition}

\usepackage{gensymb}
\usepackage{amssymb}

\newcommand{\RR}{\mathbb{R}^2}
\newcommand{\dist}{\mathop{\mathrm{dist}}}
\newcommand{\z}{\mathbb{Z}}

\usepackage{tikz}
\usetikzlibrary{decorations.pathreplacing}
\usetikzlibrary{calc}

\begin{document}
\title{Online coloring of disk graphs}

\author{
 Joanna Chybowska-Sokół$^1$\thanks{Supported by the National Science Center of Poland under Grant No. 2016/23/N/ST1/03181} \and Konstanty Junosza-Szaniawski$^1$
\\            $^1$ Faculty of Mathematics and Information Science,\\ Warsaw University of Technology, Poland\\
              email: {j.sokol@mini.pw.edu.pl, konstanty.szaniawski@pw.edu.pl}       
}

\maketitle
             
keywords: disk graphs, online coloring, online $L(2,1)$-labeling, online coloring geometric shapes 

\begin{abstract}
In this paper, we give a family of online algorithms for the classical coloring problem of intersection graphs of discs with bounded diameter. Our algorithms make use of a geometric representation of such graphs and are inspired by an algorithm of Fiala {\em et al.}, but have better competitive ratios. The improvement comes from using two techniques of partitioning the set of vertices before coloring them. One of which is an application of a $b$-fold coloring of the plane. The method is more general and we show how it can be applied to coloring other shapes on the plane as well as adjust it for online $L(2,1)$-labeling.
\end{abstract}

\section{Introduction}
Intersection graphs of families of geometric objects attracted much attention of researchers both for their theoretical properties and practical applications (c.f. McKee and McMorris \cite{McKMcM}). For example intersection graphs of families of discs, and in particular discs of unit diameter (called {\em unit disk intersection graphs}), play a crucial role in modeling radio networks. 
Apart from the classical coloring, other labeling schemes such as $T$-coloring and distance-constrained labeling of such graphs are applied to frequency assignment in radio networks \cite{hale,
roberts}. In this paper, we consider the classical coloring.

We say that a graph coloring algorithm is {\em online} if the input graph is not known a priori, but is given vertex by vertex (along with all edges adjacent to already revealed vertices). Each vertex is colored at the moment when it is presented and its color cannot be changed later. On the other hand, {\em offline} coloring algorithms know the whole graph before they start assigning colors. The online coloring can be much harder than offline coloring, even for paths. For an offline coloring algorithm 
 by the {\em approximation ratio} we mean the worst-case ratio of the number of colors used by this algorithm 
to the chromatic number of the graph 
 For online algorithms, the same value is called {\em competitive ratio}. 


A unit disk intersection graph $G$ can be colored offline in polynomial time with $3\omega(G)$ colors \cite{peeters} (where $\omega(G)$ denotes the size of a maximum clique) and online with $5\omega(G)$ colors. The last result comes from a combination of results from \cite{capponi}, which states that First-Fit algorithm applied to $K_{1,p}$-free graphs has competitive ratio at most $p-1$ with results in   \cite{malesinska,peeters} saying that unit disc intersection graphs are $K_{1,6}$-free.
  
Fiala, Fishkin, and Fomin \cite{FFF} presented a polynomial-time online algorithm that finds an $L(2,1)$-labeling of an intersection graph of discs of bounded diameter. The $L(2,1)$-labeling asks for a vertex labeling with non-negative integers, such that adjacent vertices get labels that differ by at least two, and vertices at distance two get different labels.  The algorithm is based on a special coloring of the plane, that resembles colorings studied by Exoo \cite{exoo}, inspired by the classical Hadwiger-Nelson problem \cite{hadwiger}. A similar idea of reserving a set of colors for upcoming vertices can be found in the paper by Kierstead and  Trotter \cite{kierstead}.  In \cite{udg} we described an algorithm for coloring unit disks intersection graph.  It is inspired by \cite{FFF}, however, we change the algorithm in such a way that a $b$-fold coloring of the plane (see \cite{ulamkowe}) is used instead of a classical coloring. For graphs with big $\omega$ such approach lets us obtain the competitive ratio below $5\omega$.
In this paper, we generalize results from \cite{udg} for intersection graphs of disks with bounded diameter. Moreover, we improve the algorithm from \cite{udg} by more uniform distribution of vertices into layers.  

Throughout the paper, we always assume that the input disk intersection graph is given along with its geometric representation. In the paper we consider a few algorithms for coloring $\sigma$-disk graphs (i.e. intersection graphs of disks with a diameter within $[1,\sigma]$) and some other geometric shapes, see Table \ref{tab-algorithms}. We use the divide and conquer approach. The basis of our colorings is a coloring of the plane (except for the BranchFF algorithm, where First-Fit is used), but we use 2 types of division of vertices of the $\sigma$-DG. It is easiest to get a grasp on how to color disks with a coloring of a plane by analyzing the SimpleColor algorithm. In terms of division, algorithms BranchFF and FoldColoring show the two methods we use. \emph{Branching} divides the vertices according to the size of the corresponding disks, while \emph{folding} focuses on the location of their centers.

All of these algorithms can be used for online coloring of $\sigma$-disk intersection graphs, hence we first introduce them in terms of disk colorings. However, their application can be broader. Some can be easily adjusted for online coloring of intersection graphs of various geometric shapes. And some can be used for $L(2,1)$-labeling of disks or other shapes, and the only change necessary is the coloring of the plane that we use. In the paper, we concentrate mostly on the method itself, more than the specific parameters of the algorithm such as competitive ratio, since they highly depend on the input. 

Again, just like for disk graphs, we assume that the geometric interpretation of a graph is given rather than the graph itself. We also assume that some parameters of the geometric structure are known in advance, such as the bounds of the disk diameters. It is crucial knowledge considering our algorithms are online and some of them are used for $L(2,1)$-labeling.

\begin{table}[ht]
\begin{tabular}{c|c|c|c}
Algorithm & branching & folding & What is needed:\\\hline
BranchFF\cite{erlebach} & Y & N & $-$\\
SimpleColor\cite{FFF} & N & N & a solid coloring of $G_{[1,\sigma]}$\\
BranchColor & Y & N & a solid coloring of $G_{[1,\textbf{2}]}$\\
FoldColor\cite{udg} & N & Y & a $b$-fold solid coloring of $G_{[1,\sigma]}$\\
FoldShadeColor & N & Y & a $b$-fold solid coloring of $G_{[1,\sigma]}$, shading $\eta$\\
BranchFoldColor & Y & Y & a $b$-fold solid coloring of $G_{[1,\textbf{2}]}$, shading $\eta$\\
\end{tabular}\caption{A short summary of the algorithms inner methods.}\label{tab-algorithms}
\end{table}

\section{Preliminaries}

We start with introducing some basic definitions, notations, and preliminary results. This section includes a method of coloring the plane, which will later be used in our algorithms.
\subsection{Notation}
 
For an integer $n$, we define $[n]:=\{1,\ldots, n\}$. For integers $n,b$ by $(n)_b$ we denote $n$ modulo $b$. By a \emph{graph} we mean a pair $G=(V,E)$, where $V$ is a finite set and $E$ is a subset of the set of all 2-element subsets of $V$. A function $c\colon V\to[k]$ is a $k$-coloring of $G=(V,E)$ if for any $xy\in E$ holds $c(x)\neq c(y)$.  By $d(u,v)$ we denote the number of edges on the shortest $u$-$v$--path in $G$.

For a sequence of disks in the plane $(D_i)_{i\in [n]}$ we define its intersection graph by $G((D_i)_{i\in [n]})=(\{v_i: i\in [n]\},E)$, where $v_i$ is the center of $D_i$ for every $i\in [n]$ and $v_iv_j\in E$ iff $D_{i}\cap D_{j}\neq \emptyset$. Since we have a sequence of disks rather than a set, we disregard the fact that two disks can have the same center and always treat vertices corresponding to different disks as separate entities. Any graph that admits a representation by intersecting disks is called a disk graph, or simply a DG. If the ratio between the largest and the smallest diameters of the disks is at most $\sigma$ then we call such graph a $\sigma$-disk graph or $\sigma$-DG, for short. We can assume that all disks in a representation of a $\sigma$-DG are $\sigma$-disks, i.e. their diameters are within $[1,\sigma]$. By {\em UDG} we mean the class of graphs that admit a representation by intersecting \emph{unit} disks. All of our algorithms require $\sigma$ to be known in advance. Hence we always assume to be given $\sigma$ and when we write disk graph we mean $\sigma$-DG.

 For a minimization online algorithm $\mathrm{alg}$, by $\mathrm{cr(alg)}$ we denote its \emph{competitive ratio}, which is the supremum of $\frac{\mathrm{alg}(G)}{\mathrm{opt}(G)}$ over all instances $G$, where $\mathrm{alg}(G)$ is the value of the solution given by the algorithm for instance $G$ and $\mathrm{opt}(G)$ is the optimal solution for instance $G$.  For the classical coloring we use the fact that any coloring requires at least $\omega(G)$ colors, where $\omega(G)$ denotes the size of the largest clique of $G$.

\subsection{Tilings and colorings of the plane}\label{sec:tilings}

Our algorithms use colorings of the euclidean plane as a base for $\sigma$-disk coloring. The coloring we use depends on the value of $\sigma$. In this section, we present a method of finding such colorings. We start with defining an infinite graph $G_{[1,\sigma]}$.
\begin{defi}
By $G_{[1,\sigma]}$ we denote a graph with $\mathbb{R}^2$ as the vertex set and edges between pairs of points at euclidean  distance within $[1,\sigma]$.
\end{defi}
A \emph{coloring} of $G_{[1,\sigma]}$ follows the regular definition of graph coloring. For a $b$-fold colorings we include the definition with our notation.


\begin{defi}
A function $\varphi=(\varphi_1,\ldots, \varphi_b)$ where $\varphi_i:\mathbb{R}^2 \to [k]$ for $i\in [b]$  is called a  {\em $b$-fold coloring of $G_{[1,\sigma]}$ with color set} $[k]$ if
\begin{itemize}
\item for any point $p\in \mathbb{R}^2$ and $i,j\in [b]$, if $i\neq j$, then $\varphi_i(p)\neq \varphi_j(p)$,
\item for any two points $p_1,p_2\in \mathbb{R}^2$ with $\dist(p_1,p_2)\in [1,\sigma]$ and $i,j\in [b]$ holds $\varphi_i(p_1)\neq \varphi_j(p_2)$.
\end{itemize}
The function $\varphi_i$ for $i \in [b]$ is called an {\em $i$-th layer} of $\varphi$.
\end{defi}
Notice that a coloring of $G_{[1,\sigma]}$ is a 1-fold coloring of $G_{[1,\sigma]}$. 

Coloring and $b$-fold colorings of $G_{[1,\sigma]}$ have been a subject of various papers (\cite{exoo},\cite{hadwiger},\cite{falconer},\cite{ulamkowe}, \cite{soifer}). The case of $\sigma=1$ is the most studied, as the problem of coloring $G_{[1,1}$ is known as Hadwiger-Nelson problem. For many years the best bounds on $\chi(G_{[1,1})$ were 4 and 7, which are quite easy to obtain. A breakthrough came in 2018, when Aubrey de Grey \cite{deGrey} proved, that  $\chi(G_{[1,1})\ge5$. Right after that, a paper by Exoo and  Ismailescu \cite{exoo2018} claimed the same result (the proofs were not the same). Since then there was a spike of interest in various colorings of the plane. The size of a minimal subgraph of $G_{[1,1}$ with chromatic number equal 5 became one of the subjects of study as well \cite{heule}. 

In our case, we only consider colorings based on specific tilings of the plane (i.e. we partition the plane into congruent shapes called tiles). To emphasize this we will call such colorings \emph{solid}. The assumption of our plane coloring being solid will be clearly stated in our theorems.

 In general, a coloring $\varphi$ of $G_{[1,\sigma]}$ is called {\em solid} if there exists a tiling such that each tile is monochromatic (so the diameter of a tile cannot exceed 1) and tiles in the same color are at distance greater than $\sigma$. A $b$-fold coloring $\varphi=(\varphi_1,\ldots,\varphi_b)$ of the plane is called {\em solid} if for all $i\in [b]$ the coloring $\varphi_i$ is solid and tiles in the same color (possibly from different tilings) are at distance greater than $\sigma$. An example of a solid 1-fold 7-coloring of the plane (by John Isbell \cite{soifer}) is presented in Figure \ref{7-kol}. 

Now we introduce our method of building solid $b$-fold colorings of the plane. Note that it is possible to construct other solid colorings that could be used for our algorithms, but finding them is not the main problem of the paper.

We define an $h^2$-fold coloring of the plane based on hexagonal tilings, for any positive integer $h$. For $h=1$ it is a coloring of the plane as presented in \cite{MYplaszczyzna} and we follow the notation from both \cite{MYplaszczyzna} and \cite{udg}. From this point onward we will assume $h$ to be a fixed positive integer and often omit it in further notation.

Let us start with defining the tiling. Let $H_{0,0}$ be a hexagon with two vertical sides, center in $(0,0)$, diameter equal one, and part of the boundary removed as in Figure \ref{heksagon}. Note that the width of $H_{0,0}$ equals $\frac{\sqrt{3}}{2}$.
Then let $s_1=[\frac{\sqrt{3}}{2},0]$, and $s_2=[\frac{\sqrt{3}}{4},-\frac{3}{4}]$. For $i,j\in \mathbb{Z}$ let $H_{i,j}$ be a \emph{tile} created by shifting $H_{0,0}$ by a vector $i\cdot\frac{s_1}{h}+j\cdot\frac{s_2}{h}$, namely $H_{i,j}:=\{(x,y)+i\cdot\frac{s_1}{h}+j\cdot\frac{s_2}{h}\in\RR:\ (x,y)\in H_{0,0}\}$.  
Notice that if $h=1$ then $\{H_{i,j}:i,j\in \mathbb{Z}\}$ forms a partition of the plane, which we call a hexagonal tiling. More generally for any $m\in [h^2]$ the set $L_m=\{H_{i,j}: 1+(i)_h+h(j)_h=m\}$ forms a tiling (see Figure \ref{tiling}), we call it the \emph{$m$-th  layer}.
\begin{figure}[h] \centering
\subfigure[A tile $H_{i,j}$.]{\label{heksagon}
\includegraphics[height=3.5cm]{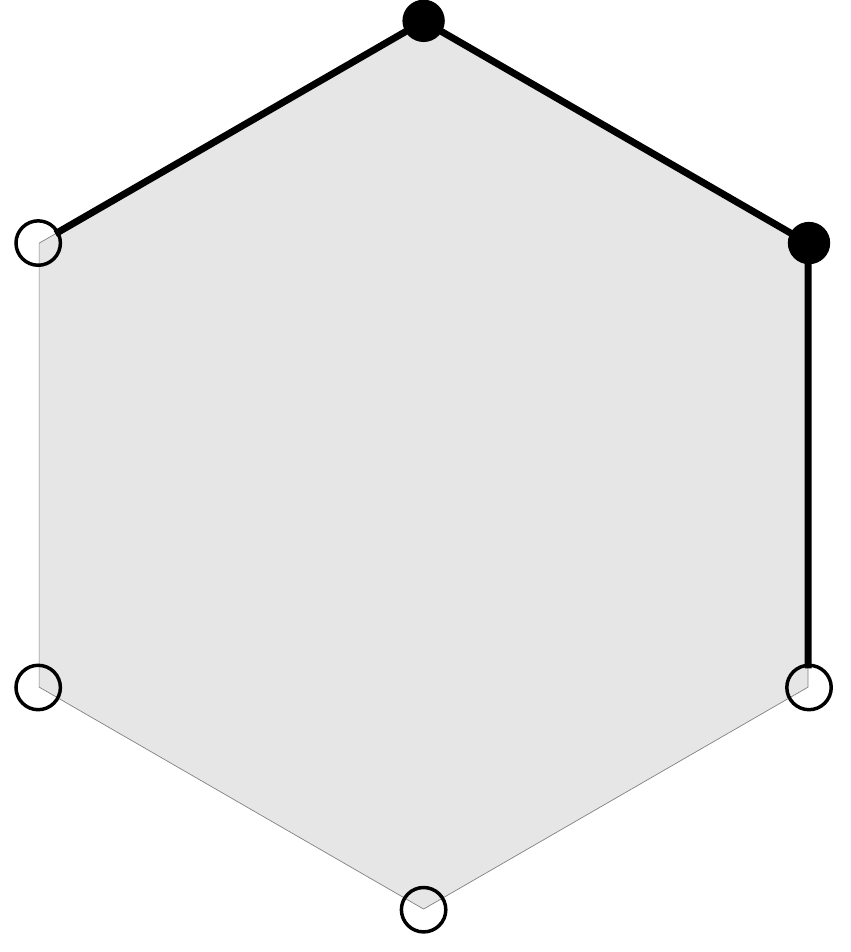}
}\hfill
\subfigure[Hexagonal tiling - the 1$^\text{st}$ layer.]{\label{tiling}\includegraphics[width=0.6\textwidth]{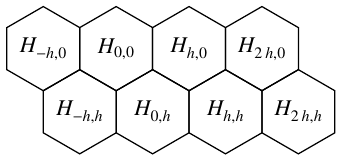}}
\hfill
\subfigure[7-coloring of the plane]{\label{7-kol}\includegraphics[width=0.8\textwidth]{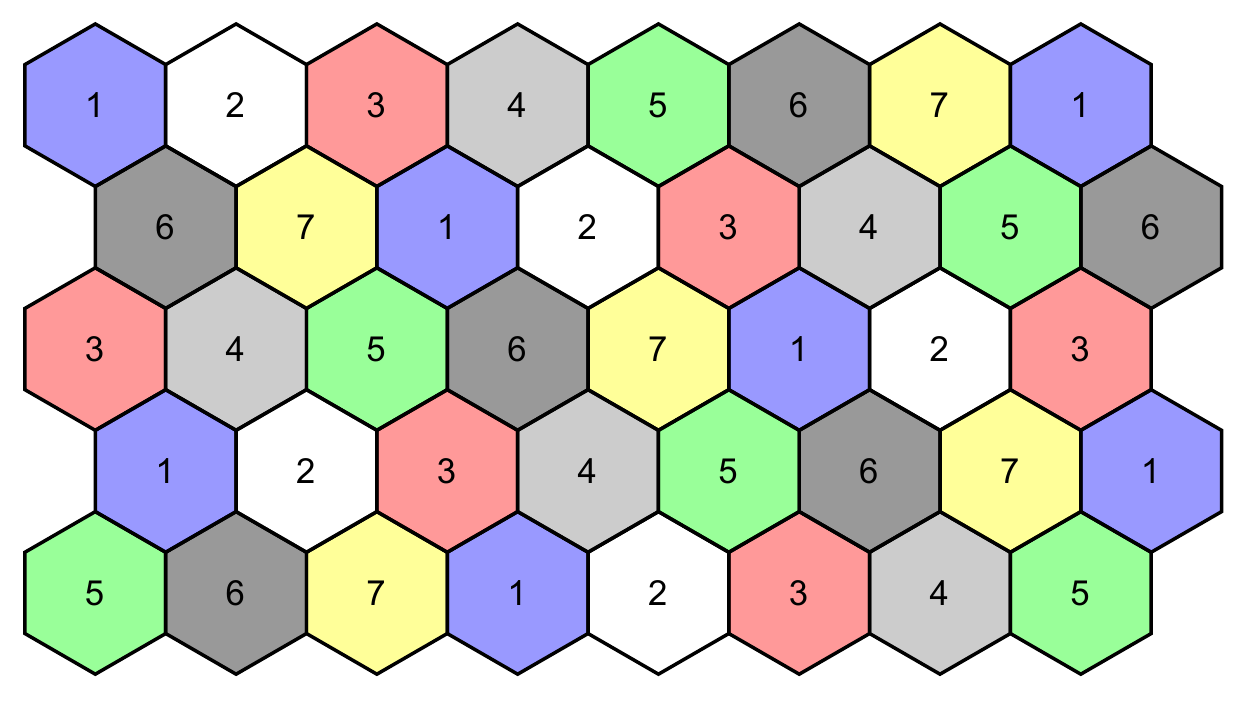}}
\caption{\cite{udg} Hexagonal tiling. }
\end{figure}

All points in a single tile will share a color. Now for a triple of non-negative integers $(h^2,p,q)$ we define a coloring such that for any $(i,j)\in \mathbb{Z}^2$ hexagons $H_{i,j}, H_{i+p,j+q},H_{i+p+q,j-p}$ have the same color. Notice that the centers of such three hexagons form an equilateral triangle (see the three hexagons on the right side of Figure \ref{pattern}). By reapplying this rule of a single colored triangles we obtain that sets of the form $\{H_{i+k\cdot p + l\cdot (p+q),j+k\cdot q-l\cdot p}:\ k,l\in \mathbb{Z}\}$ are monochromatic. When maximal monochromatic sets are of this form we call such coloring a $(h^2,p,q)$-coloring.

\begin{lemma}\label{lem:pqcolors}
A $(h^2,p,q)$-coloring uses $p^2+p q+q^2$ colors.
\end{lemma}
\begin{proof}
Let us denote greatest common divisor of $p$ and $q$ as $d$ and let $p'=p/d$, $q'=q/d$.

We call the set of $H_{i,j}$, $i\in\mathbb{Z}$, the $j$-th row of tiles. Let us call the color of $H_{0,0}$ blue and $\mathcal{T}=\{H_{k\cdot p + l\cdot (p+q),k\cdot q-l\cdot p}:\ k,l\in \mathbb{Z}\}$ denote the set of all blue hexagons. To give some insight of where $\mathcal{T}$ comes from let $v$ denote a vector from $(0,0)$ to the center of $H_{p,q}$ ($v=p\cdot \frac{s_1}{h}+q\cdot \frac{s_2}{h}$) and $\overline{v}$ be a vector obtained from $v$ via rotating it by $\frac{\pi}{3}$ ($\overline{v}=p\cdot \frac{s_1 - s_2}{h}+ q\cdot \frac{s_1}{h}=(p+q)\cdot \frac{s_1}{h}-p\cdot \frac{s_2}{h}$) - see Figure \ref{pattern}. Then $\mathcal{T}$ is a set of all tiles created by shifting $H_{0,0}$ by $k v + l \overline{v}$, for $k,l\in \mathbb{Z}$.
\begin{figure}[h]
\center
\includegraphics[scale=1]{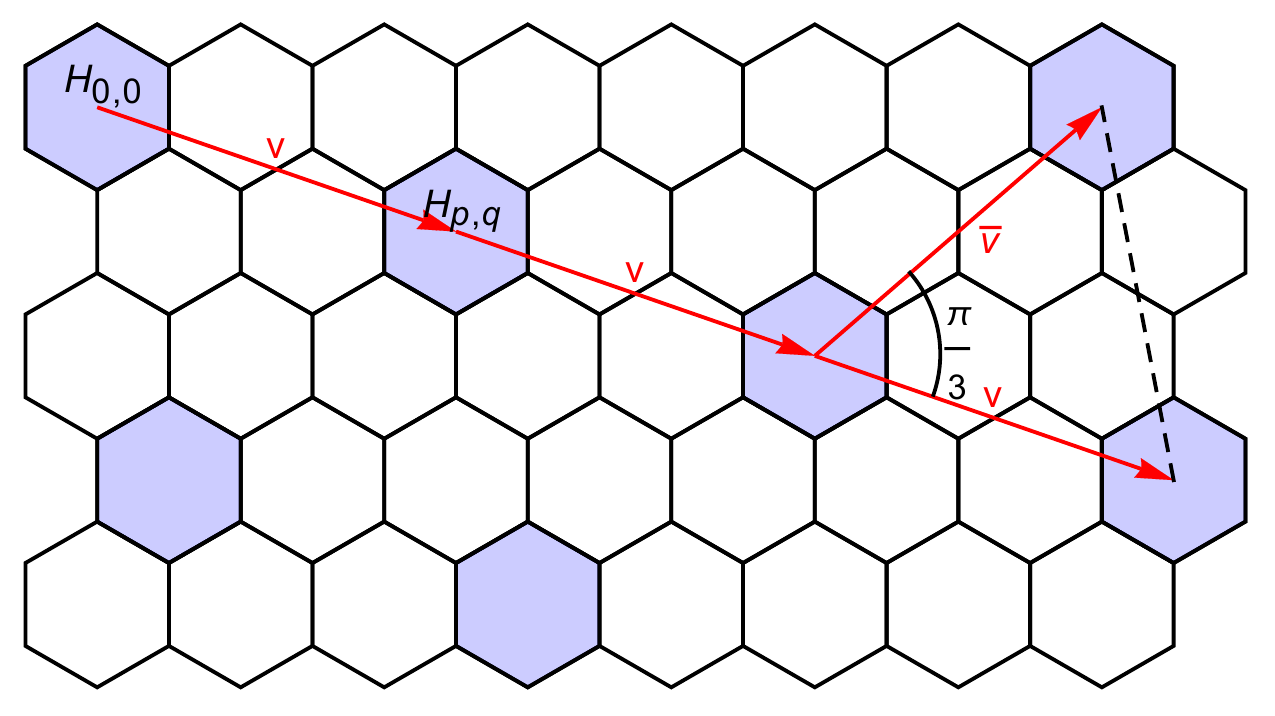}\caption{A single color in a $(1,2,1)$-coloring. The angle between $v$ and $\overline{v}$ does not depend on the values of $h,p,q$.}\label{pattern}
\end{figure}

First let us notice that, by the definition of $\mathcal{T}$, blue appears in rows numbered by $k\cdot q - l\cdot p=(kq'- lp')\cdot d$, for any $k,l\in \mathbb{Z}$. 
Since $p'$ and $q'$ are coprime, $\{kq'- lp':\ k,l\in \mathbb{Z}\}=\mathbb{Z}$.
 Then blue appears in a row iff its number is divisible by $d$.

Note that, since we used the same pattern (only shifted) for every color, we have the same number of colors in every row. Hence the total number of colors equals the number of colors used in a single row multiplied by $d$.

To find the number of colors used in a row, it is enough to know how often does a single color reappear in it (see example in Figure \ref{7-kol} in one row every 7$^{\text{th}}$ hexagon is blue and we use 7 colors in each row). 


Let $m$ be the smallest positive number such that $H_{m,0}$ is colored blue. Since $H_{m,0}\in\mathcal{T}$, then there exist integers $k,l$ such that:\begin{align}
k\cdot p + l\cdot (p+q)=m,\label{e1}\\
k\cdot q-l\cdot p=0.\label{e2}
\end{align}

From (\ref{e2}) we derive $kq'=lp'$. Since $p',\ q'$ are coprime, then $k$ is divisible by $p'$ and $l$ by $q'$. Moreover $p',\ q'$ are both non-negative, hence either $k$ and $l$ are both non-positive or they are both non-negative. Without loss of generality we assume the latter.
Hence $m=d\cdot(k p' + l q' +l p')$ is minimal when $k=p'$ and $l=q'$. So the numbers of colors in a row equals $m=d(p'p' + q'q' +p'q')$, and from our previous remarks we conclude that the total number of colors equals $d^2(p'p' + q'q' +p'q')=p^2+q^2+pq$.
\end{proof}

In case $q=0$ we can easily express the number of colors we use in terms of $\sigma$, which was done in \cite{ulamkowe} (see Figure \ref{hkw}). It is enough to take $p=\lceil (\frac{2\sigma}{\sqrt{3}}+1)\cdot h \rceil$, as then the distance between centers of $H_{0,0}$ and $H_{p,0}$ equals $|p\cdot\frac{s1}{h}|=\lceil (\frac{2\sigma}{\sqrt{3}}+1) h \rceil\cdot \frac{\sqrt{3}}{2h}\ge \sigma+\frac{\sqrt{3}}{2}$. Since these tiles are in the same row, then the distance between them is in fact the distance between their vertical sides and equals the distance between the centers minus $\frac{\sqrt{3}}{2}$.

\begin{prop}[\cite{ulamkowe}]\label{prop:hkwadrat}
For $h\in \mathbb{N_+}$ there exists a solid $h^2$-fold coloring $\varphi$ of $G_{[1,\sigma]}$ with $\left \lceil (\frac{2\sigma}{\sqrt{3}}+1)\cdot h \right \rceil^2$ colors.
\end{prop}
\begin{figure}[h]
\includegraphics[width=\textwidth]{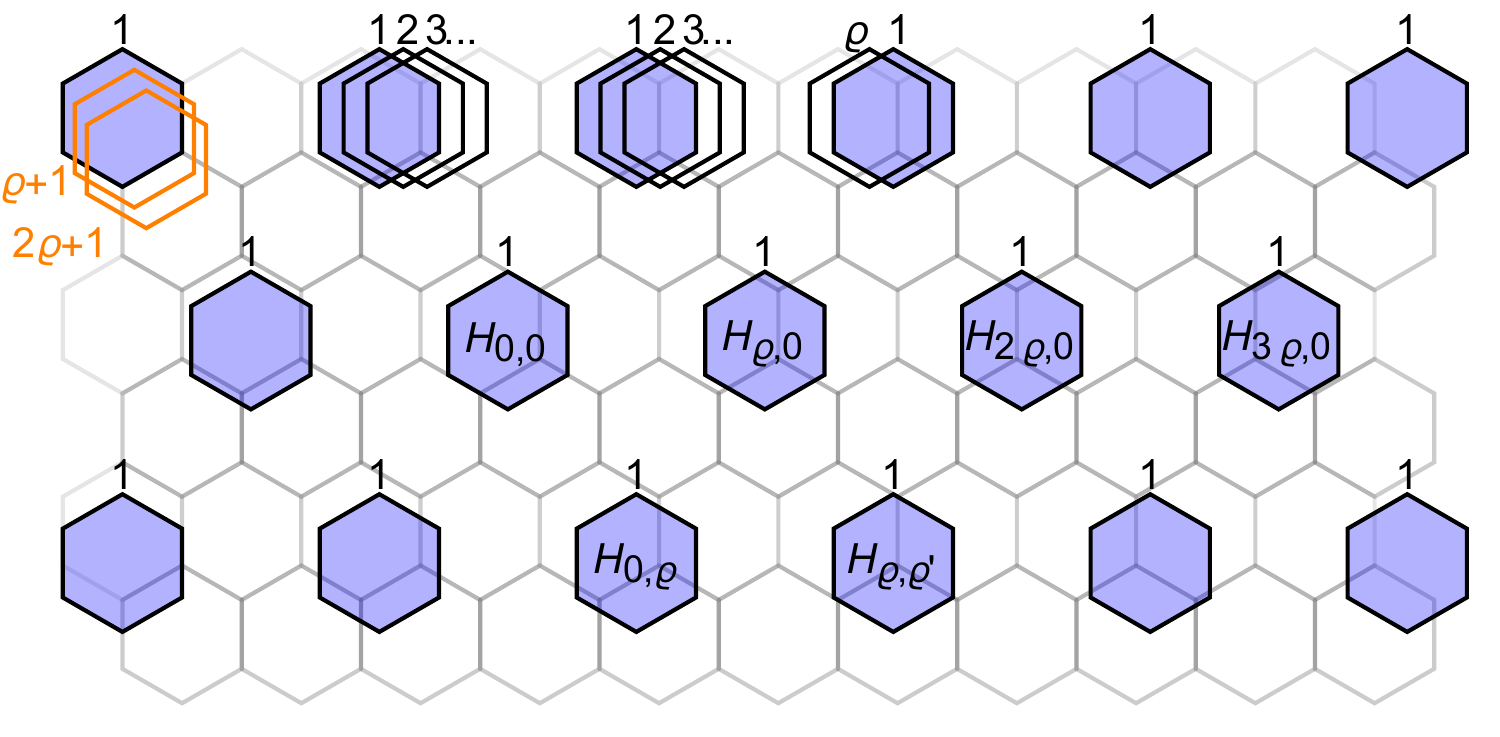}
\caption{\cite{udg} $h^2$-fold coloring of the plane by Grytczuk, Junosza-Szaniawski, Sokół, Węsek}\label{hkw}
\end{figure}

In more general case it is a bit hard to find the maximal $\sigma$ such that the $(h^2,p,q)$-coloring is a coloring of $G_{[1,\sigma]}$. $\sigma$ for $(h^2,p,q)$-coloring.
For some values of $(h^2,p,q)$ such $\sigma$ might not exist at all, since the distance between two tiles of the same color could be smaller than 1.  It is enough to consider the distance between $H_{0,0}$ and $H_{p,q}$, but this value depends on which points in these two tiles minimize the distance. When computing the maximal sigma we consider the distances between any two points on the boundaries of $H_{0,0}$ and $H_{p,q}$. The minimal distance we find is the maximal value of However we can easily find some values of $\sigma$ for which the $(h^2,p,q)$-coloring works. 
\begin{prop}\label{prop:pqsigma}
For any fixed values of $h,p,q$, let $\sigma=\frac{\sqrt{3}}{2h}\sqrt{p^2+pq+q^2}-1$. If $\sigma\ge1$, then the $(h^2,p,q)$-coloring is a coloring of $G_{[1,\sigma]}$.
\end{prop}
\begin{proof}
Let as consider the minimal distance between two points $P_1\in H_{0,0}$ and $P_2\in H_{p,q}$. Let us denote the centers of these tiles as $C_1$ and $C_2$. By the definition of $H_{i,j}$, the distance between $C_1$ and $C_2$ equals $|p\cdot \frac{s_1}{h} +q\cdot \frac{s_2}{h}|=
|\frac{1}{h}[p\frac{\sqrt{3}}{2}+q\frac{\sqrt{3}}{4},-q\frac{3}{4}]|
=
\frac{1}{h}\sqrt{((2p+q)\frac{\sqrt{3}}{4})^2+(-q\frac{3}{4})^2}=\frac{\sqrt{3}}{2h}\sqrt{p^2+pq+q^2}$. The distance between $P_1$ and $C_1$ is at most $\frac{1}{2}$, and the same is true for $P_2$ and $C_2$. Hence the distance between any $P_1\in H_{0,0}$ and $P_2\in H_{p,q}$ is at most $\frac{\sqrt{3}}{2h}\sqrt{p^2+pq+q^2}-1=\sigma\ge 1$.

Any two points of the same color are either both in one tile - then their distance is smaller than $1$, or in two tiles of the same color. The distance between any two tiles of the same color is at least as big as the distance between $H_{0,0}$ and $H_{p,q}$ (by the construction of $(h^2,p,q)$-coloring). Hence in the latter case the distance between such points is at least $\frac{\sqrt{3}}{2h}\sqrt{p^2+pq+q^2}-1$, which concludes the proof.
\end{proof}

In most cases it is possible to use $(h^2,p,q)$ colorings for $\sigma$ larger than stated above, but the Proposition shows the approximate value. Now by Lemma \ref{lem:pqcolors} and Proposition \ref{prop:pqsigma} we obtain the following.
\begin{cor}
For any fixed values of $h,p,q$, let $k=p^2+pq+q^2$ and $\sigma=\frac{\sqrt{3}}{2h}\sqrt{k}-1$. If $\sigma\ge1$, then the $(h^2,p,q)$-coloring is an $h^2$-fold $k$-coloring of $G_{[1,\sigma]}$.
\end{cor}
 By this corollary we get that $k\approx\frac{4h^2}{3}(\sigma+1)^2$. This is a worse result then the one from Proposition \ref{prop:hkwadrat}. However if we consider the maximal values of $\sigma$ for some cases $(h^2,p,q)$-colorings we get better results. In particular one of the best $h^2$-fold colorings of $G_{[1,2]}$ and 'small' values of $h$ is the $(8^2,1,26)$-coloring, which uses 703 colors. More examples of precise values of $\sigma$ are presented in Table \ref{tab:hpq}. It contains only the best $(h^2,p,q)$-colorings for $h\in\{1,2,3\}$ and $\sigma$ up to 3.
 \begin{table}[h]
 \begin{tabular}{c|c|c|c|c|c}
 $\sigma$&$\frac{k}{h^2}$&$h^2$&$k$&$p$&$q$\\\hline
  1.01036& 4.77778&9&43&1&6\\
 1.08253& 5.25&4&21&1&4\\
 1.1547& 5.44444&9&49&0&7\\
 1.29904& 6.33333&9&57&1&7\\
 1.32288& 7.&9&63&3&6\\
 1.44338& 7.11111&9&64&0&8\\
 1.51554& 7.75&4&31&1&5\\
 1.58771& 8.11111&9&73&1&8\\
 1.60728& 8.77778&9&79&3&7\\
 1.63936& 9.25&4&37&3&4\\
 1.73205& 9.33333&9&84&2&8\\
 1.75& 9.75&4&39&2&5\\
 1.87639& 10.1111&9&91&1&9\\
 1.94856& 10.75&4&43&1&6\\
 2.02073& 11.1111&9&100&0&10\\
 2.02073& 12.1111&9&109&5&7\\
  \end{tabular}
  \hspace{0.5cm}
 \begin{tabular}{c|c|c|c|c|c}
 $\sigma$&$\frac{k}{h^2}$&$h^2$&$k$&$p$&$q$\\\hline
 2.04634& 12.25&4&49&3&5\\
 2.16506& 12.3333&9&111&1&10\\
 2.17945& 13.&9&117&3&9\\
 2.3094& 13.4444&9&121&0&11\\
 2.38157& 14.25&4&57&1&7\\
 2.45374& 14.7778&9&133&1&11\\
 2.46644& 15.4444&9&139&3&10\\
 2.59808& 16.3333&9&147&2&11\\
 2.61008& 16.75&4&67&2&7\\
 2.64575& 17.3333&9&156&4&10\\
 2.74241& 17.4444&9&157&1&12\\
 2.75379& 18.1111&9&163&3&11\\
 2.81458& 18.25&4&73&1&8\\
 2.88675& 18.7778&9&169&0&13\\
 2.92973& 20.1111&9&181&4&11\\
 3.03109& 20.3333&9&183&1&13\\
 \end{tabular}\caption{A few records of $(h^2,p,q)$-coloring. $k$ denotes the number of colors and is equal $p^2+pq+q^2$.}\label{tab:hpq}
 \end{table}

%

The last thing we should know about our colorings is the number of subtiles in a tile. Having $b$ tilings (in our case $b=h^2$), by a {\em subtile} we mean a non-empty intersection of $b$ tiles, one from each layer. By $\gamma=\gamma(\varphi)$ we denote the maximum number of subtiles in a single tile in a solid $b$-fold coloring of the plane $\varphi$.

\begin{lemma}[\cite{udg}]\label{gamma}
The number of subtiles in any tile in the set $\{H_{i,j}: i,j\in\z\}$ is equal to: $\gamma =1$ for $h=1$, $\gamma =12$ for $h=2$, $\gamma =6h^2$ for $h>2$.
\end{lemma}

Notice that for any $h$, the number above is no larger than $6h^2$. In the next chapters, we will usually write about $b$-fold colorings, rather than $h^2$-fold colorings, since one could apply different colorings, then mentioned in this section. However, when the bound on the number of colors is given, substituting $\gamma$ for $6b$ might give a bit of extra insight.
\subsection{$L^*(2,1)$-labeling of the plane}\label{sec:planeL21}

\begin{defi}\label{def:Lplane}
A solid $b$-fold $k$-coloring $\varphi$ of $G_{[1,\sigma]}$ is called a {\em solid $b$-fold $L^*(2,1)$-labeling of $G_{[1,\sigma]}$ with $k$ labels}, if 
\begin{enumerate}
\item all labels (colors) are in $\{1,2,\ldots,k\}$, or all are in  $\{0,1,\ldots k-1\}$,
\item for any two tiles $T_1,T_2$  with the same color, $T_1$ and $T_2$ are at point-to-point distance greater than $2\sigma$. 
\item for any two tiles $T_1,T_2$  with consecutive  colors, $T_1$ and $T_2$ are at point-to-point distance greater than $\sigma$
\item for a tile $T_1$ with the smallest color and a tile $T_2$ with the largest color, $T_1$ and $T_2$ are at point-to-point distance greater than $\sigma$.
\end{enumerate}
\end{defi}

Note that by this definition any two points of the same tile must still be lower than $1$.

Such labeling were briefly considered in \cite{FFF} (called 'circular labeling') and more thoroughly in case of $\sigma=1$ in \cite{udg}. Fiala, Fishkin, and Fomin state in \cite{FFF} that $L^*(2,1)$-labelings can always be found, give the estimated time of searching, and give an example of such coloring for $\sigma=\frac{\sqrt{7}}{2}$. Grytczuk, Junosza-Szaniawski, Sokół, and Węsek in \cite{udg} concentrate on finding such $b$-fold labelings for $\sigma=1$, and in particular, give a method of constructing them in case $b=h^2$.

\begin{theorem}[\cite{udg}]\label{th:hkwadratL21}
For all $h\in \mathbb{N}$, there exists a solid $h^2$-fold $L^*(2,1)$-labeling of $G_{[1,1]}$ with $k=3\lceil h(\frac{2}{\sqrt{3}}+1)+1\rceil^2+1$ labels.
\end{theorem}
\begin{figure}[h]\begin{center}
\includegraphics[width=\textwidth]{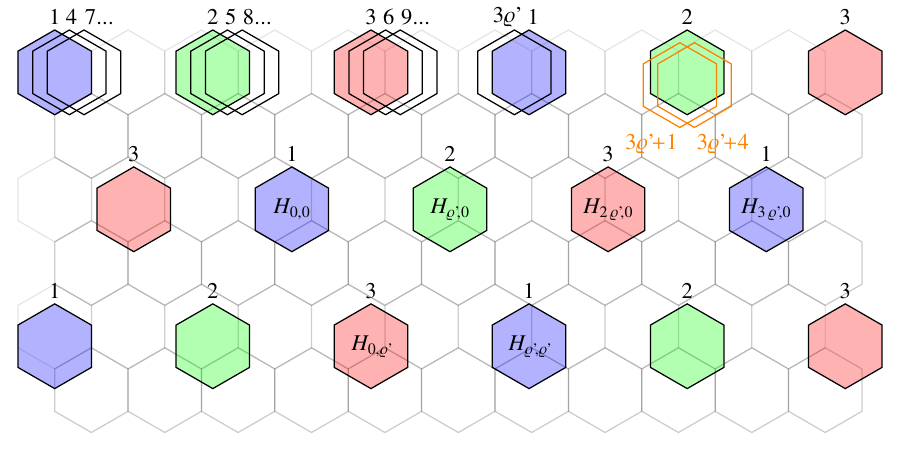}
\caption{\cite{udg} A solid $h^2$-fold $L^*(2,1)$-labeling of $G_{[1,1]}$}
\label{fig:l21G11}
\end{center}\end{figure}
The construction is quite similar to the one from Proposition \ref{prop:hkwadrat}. The key idea is to substitute a single color of $(h^2,p,0)$-coloring, where $p=\lceil h(\frac{2}{\sqrt{3}}+1)+1\rceil$, with 3 consecutive labels (see Figure \ref{fig:l21G11}). Tiles of such triplets of labels are monochromatic in $(h^2,p,0)$-coloring, hence their distances are bigger than 1. The tiles with the same label are at distance larger than 2. The fact that $p$ is larger than $\lceil h(\frac{2}{\sqrt{3}}+1)\rceil$ by 1 is enough to make sure that tiles with consecutive labels from different triplets are at distance greater than 1 (see labels $3\varrho-1$ and $3\varrho$ in Fig. \ref{fig:l21G11}). Naturally, in such construction, we use labels from 1 to $p^2$, but the declared number of labels is $p^2$+1, hence there is no conflict between the smallest and largest labels.

 We can do the same for $\sigma\in [1,\frac{1}{4-2\sqrt{3}}]$ as well. The upper bound on $\sigma$ here comes from the restriction on tiles with the same label (by the second condition of the definition, their distance must be greater than $2\sigma$). The centers of such tiles are at distance
$|p\cdot\frac{s_1}{h}+p\cdot\frac{s_2}{h}|=\frac{p}{h}|s_1+s_2|=\frac{3p}{2h}$ and we need this value to be greater than $2\sigma+1$. 
  
\begin{prop}\label{prop:planesigmaL21} If $1\le\sigma\le \frac{1}{4-2\sqrt{3}}\approx 1.86603$, then
for all $h\in \mathbb{N}$, there exists a solid $h^2$-fold $L^*(2,1)$-labeling of $G_{[1,\sigma]}$ with $k=3\lceil h(\frac{2\sigma}{\sqrt{3}}+1)+1\rceil^2+1$ labels.
\end{prop}
\begin{figure}[ht]\begin{center}
\includegraphics[width=\textwidth]{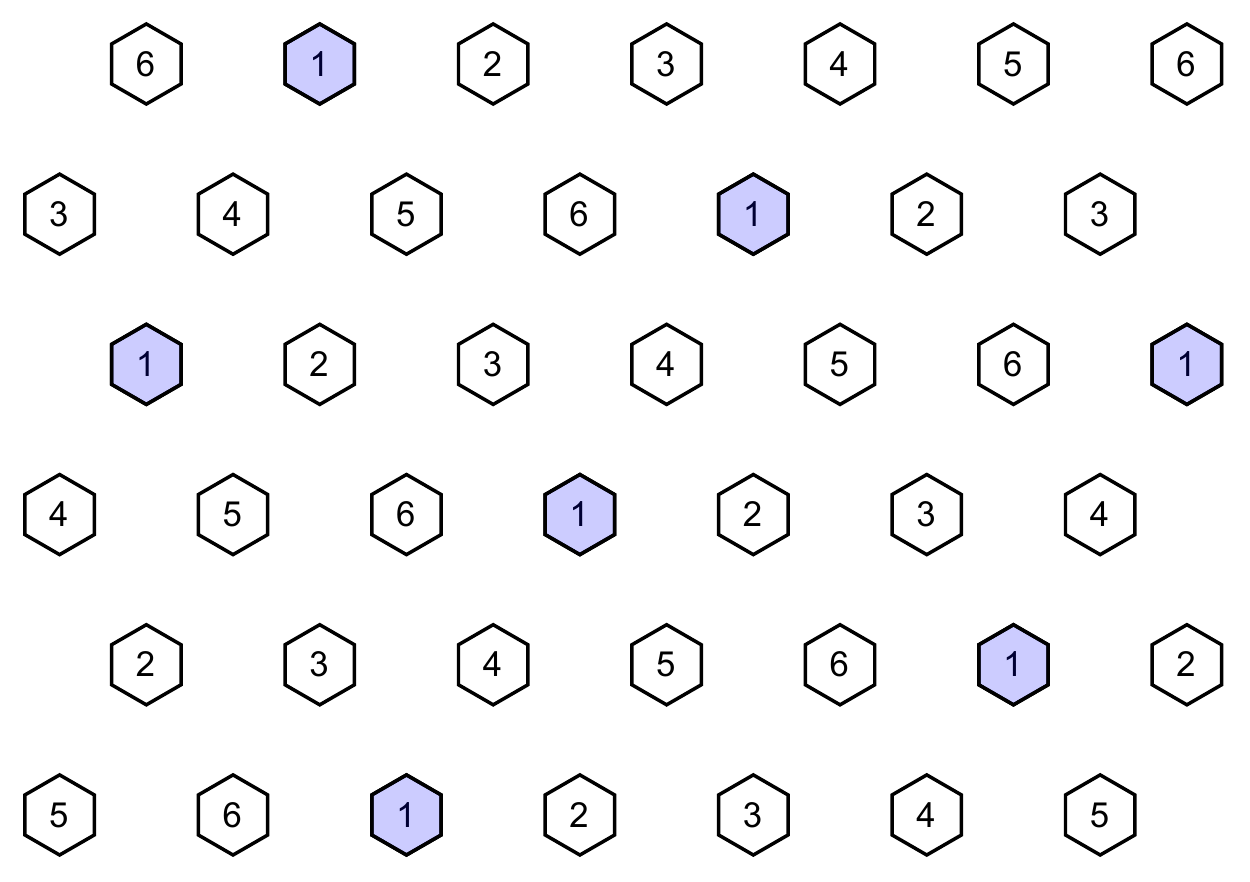}
\caption{A solid $h^2$-fold $L^*(2,1)$-labeling of $G_{[1,\sigma]}$}
\label{fig:l21slabe}
\end{center}\end{figure}
The issue of the distance between two tiles of the same label could be fixed in two ways. The first one is to choose a larger $p$. The other is to simply label monochromatic sets with 6 consecutive labels as in Figure \ref{fig:l21slabe}. It is easy to notice that in this case, the distances between tiles of the same label are at least $2\sigma+\frac{\sqrt{3}}{2}>2\sigma$.

\begin{cor}\label{cor:planel21slabe} For all $h\in \mathbb{N}$, there exists a solid $h^2$-fold $L^*(2,1)$-labeling of $G_{[1,\sigma]}$ with $k=6\lceil h(\frac{2\sigma}{\sqrt{3}}+1)+1\rceil^2+1$ colors.
\end{cor}

Increasing $p$ and labeling monochromatic sets with 3 labels should give a smaller number of labels $k$, but the full analysis of correctness is quite tedious. For now, we are content with Corollary \ref{cor:planel21slabe}, since our main aim is to show that some of our online coloring algorithms can be used for $L(2,1)$-labeling disk graphs. The important fact to note is that a $b$-fold $L^*(2,1)$-labelings of $G_{[1,\sigma]}$ with $k$ labels, can have smaller ratio $k/b$ then the number of labels in  $1$-fold $L^*(2,1)$-labeling.

It is also worth considering whether there is a more general method of constructing solid $b$-fold $L^*(2,1)$-labelings of $G_{[1,\sigma]}$, much like $(h^2,p,q)$-colorings. While it is possible to find a method that is more general than what we have shown, it should still have some rigid details. One way to do that is to keep labeling monochromatic sets of $(h^2,p,q)$-colorings with multiple labels, but without setting $q=0$.

\section{Basic coloring algorithms}
The algorithms are online so they color the vertex $v_i$ corresponding to the disk $D_i$ knowing only disks $D_1,\ldots, D_i$. Once the vertex $v_i$ is colored it cannot be recolored.

First let us introduce the algorithm by Erlebach and Fiala \cite{erlebach} (with a slight change in assignment of $j$, which will be described below), which divides the set of vertices of a disk graph according to the diameters of the corresponding disks.

\begin{algorithm}[H]\label{Alg:BranchFF}
\caption {$BranchFF((D_i)_{i\in [n]}$)}
\ForEach {$i\in [n]$}{Read $D_i$, let $v_i$ be the center of $D_i$\\
\eIf{$\sigma_i=\sigma=2^t,\ t\in\mathbb{Z}_{+}$}{$j\gets\lfloor\log_2\sigma_i\rfloor-1$}
{$j\gets\lfloor\log_2\sigma_i\rfloor$}
$B_j := B_j \cup \{v_i\}$\\
$F := \{c(v_k): 1 \leq k < i, v_k \in B_j, D_k \cap D_i \neq \emptyset \} 
\cup\ \{c(v_k): 1 \leq k < i, v_k \notin L_j \}$\ (the set of forbidden colors)\\
$c(v_i) := min(N \setminus F )$
}
\Return {$c$}
\end{algorithm}
Notice that the algorithm uses a separate set of colors for each set $B_j$.
Since for any $j$ the set $B_j$ consists of disks with diameters within $[2^j,2^{j+1}]$, so the ratio of the largest and the smallest diameters is at most 2. This holds in particular in the case of $\sigma=2^t,\ t\in\mathbb{Z}_{+}$, when the maximal $j$ equals $t-1$ and $B_{t-1}$ consists of disks of diameter $[2^{t-1},2^t]$. The unusual assignment of $j$ in such cases lets us avoid creating a separate set $B_t$ consisting solely of disks with a diameter equal $sigma$. Now we can apply the following.
 
\begin{lemma}[\cite{erlebach}] The FirstFit coloring algorithm is 28-competitive on disks of diameter ratio bounded by two.\end{lemma}

Since $j$ takes values from $0$ to $\lfloor\log_2\sigma\rfloor$ (or $\lfloor\log_2\sigma\rfloor-1$ for $\sigma=2^t$), there are at most $\lceil\log_2\sigma\rceil$ sets $B_j$ with distinct sets of colors. Each of $B_j$ is colored with no more than $28\chi_j\leq28\chi(G)$, where $\chi_j$ stands for the minimal number of colors required to color a graph induced by $B_j$. Hence the following stands.
\begin{lemma}\label{lemBranchFF}
$BranchFF$ is $28\lceil\log_2\sigma\rceil$-competitive for $\sigma$-disk graphs.\end{lemma}

Now we present another way of dividing vertices of a disk graph by Fiala, Fishkin, and Fomin \cite{FFF}. The input of the algorithm is a solid coloring $\varphi$ of $G_{[1,\sigma]}$ and a sequence of the disks of diameters within $[1,\sigma]$.
The idea of the algorithm is the following: find the tile of $\varphi$ containing the center of $D_i$. The color of a vertex corresponding to $D_i$ is the smallest available color equal modulo $k$ to the color of the tile containing the center of $D_i$.

\begin{algorithm}[H]\label{Alg:SimpleColor}
\caption {$SimpleColor_\varphi$($(D_i)_{i\in [n]}$)}
\ForEach {$i\in [n]$}{Read $D_i$, let $v_i$ be the center of $D_i$\\
 let $T(v_i)$ be the tile containing $v_i$\\
$t(v_i) \gets |\{v_1,\ldots v_{i-1}\}\cap T(v_i)|$\\
$c(v_i)\gets \varphi (v_i)+k\cdot t(v_i)$ \label{line:c}
}
\Return {$c$} \label{lineNo}
\end{algorithm}

The following theorem is a special case of a result from \cite{FFF}. It is also a special case of Theorem \ref{warstwycol} of this paper. The number of colors follows from the fact that all vertices from a single tile create a clique, hence $t(v_i)\leq \omega(G)-1$, for any $i\in [n]$. 

\begin{theorem}[Fiala, Fishkin, Fomin \cite{FFF}]
Let $\varphi$ be a solid $k$-coloring of $G_{[1,\sigma]}$ and  $(D_i)_{i\in [n]}$ sequence of $\sigma$-disks. Algorithm $SimpleColor_\varphi((D_i)_{i\in [n]})$ returns  coloring of $G=G((D_i)_{i\in [n]})$ with at most $k\cdot \omega(G)$ colors.
\end{theorem}

\begin{cor} For a unit disk graph $G=G((D_i)_{i\in [n]})$ and a 7-coloring of the plane $\varphi$ from the picture \ref{7-kol} the coefficient ratio of the algorithm $SimpleColor_\varphi((D_i)_{i\in[n]})$ is 7. 
\end{cor}
\begin{proof}
$\mathrm{cr}(SimpleColor_\varphi( (D_i)_{i\in[n]})\le \frac{7\omega(G)}{\omega(G)}=7$.
\end{proof}
\begin{cor} For a $\sigma$-disk graph $G=G((D_i)_{i\in [n]})$ and a coloring of $G_{[1,\sigma]}$ from Proposition \ref{prop:hkwadrat} the coefficient ratio of the algorithm $SimpleColor_\varphi((D_i)_{i\in[n]})$ is $\lceil\frac{2\sigma}{\sqrt{3}}+1\rceil^2$. 
\end{cor}
In the case of UDGs, the competitive ratio of $SimpleColoring$ is worse than that of FirstFit algorithm. In the next part of the paper, we consider more complex algorithms with a better ratio that are based on the same principles as $BranchFF$ and $SimpleColor$.

The next algorithm is a combination of the two above. Just like the $BranchFF$ it divides vertices according to the diameters but instead of coloring every $B_j$ with the smallest safe color, much like the FirstFit algorithm, we color them according to $SimpleColor$ algorithm. First we need a solid $12$-coloring $\varphi$ of $G_{[1,2]}$. Then for any $j\in [0,log_2(\sigma)]\cap\mathbb{N}$ we define $\varphi_j$ to be a coloring of $G_{[2^j,2^{j+1}]}$ that is scaled copy of $\varphi$. The color of a vertex is a pair $(j,c)$ where $j$ is the number assigned by branching and $c$ is the color based on $\varphi_j$. It could easily be changed so that the color is a number, but in the case of graph coloring we only consider how many colors are used and we believe leaving the color as a pair makes the method more clear.

\begin{algorithm}[H]\label{Alg:BranchColor}
\caption {$BranchColor_{(\varphi_j)}((D_i)_{i\in [n]}$)}
\ForEach {$i\in [n]$}{Read $D_i$, let $v_i$ be the center of $D_i$\\
\eIf{$\sigma_i=\sigma=2^t,\ t\in\mathbb{Z}_{+}$}{$j(v_i)\gets\lfloor\log_2\sigma_i\rfloor-1$}
{$j(v_i)\gets\lfloor\log_2\sigma_i\rfloor$}
$B_j := B_j \cup \{v_i\}$\\
let $T(v_i)$ be the tile of $\varphi_j$ containing $v_i$\\
$t(v_i) \gets |\{v_1,\ldots v_{i-1}\}\cap T(v_i)|$\\
$c(v_i)\gets (j,\varphi_j (v_i)+12\cdot t(v_i))$
\label{Blinej:c}}
\Return {$c$}
\end{algorithm}

\begin{theorem}\label{thBranchColor}
Let $(\varphi_j)$ be a sequence of solid $12$-colorings of $G_{[2^j,2^{j+1}]}$ for $j\in [0,log_2(\sigma)]\cap\mathbb{N}$, and $(D_i)_{i\in [n]}$ sequence of $\sigma$-disks. Algorithm $BranchColor_{(\varphi_j)}((D_i)_{i\in [n]})$ returns a coloring of $G=G((D_i)_{i\in [n]})$ with at most $12\lceil\log_2(\sigma)\rceil\cdot \omega(G)$ colors.
\end{theorem}
\begin{proof}
First let us prove that $c$ is indeed a coloring of $G$. Consider any two centers of disks $v_{i1}$ and $v_{i2}$ at euclidean distance at most $\sigma$. If $j(v_{i1})\neq j(v_{i2})$ then the first terms of $c(v_{i1})$ and $c(v_{i2})$ differ. Otherwise $j(v_{i1})=j(v_{i2})=:j$ and the distance between $v_{i1}$ and $v_{i2}$ is at most $2^{j+1}$. If $T_{j}(v_{i1})=T_{j}(v_{i2})$ then $\varphi_j(v_{i1})=\varphi_j(v_{i2})$ but 
$t(v_i)\neq t(v_j)$, and by formula in line \ref{Blinej:c} we get $c(v_i)\neq (c_j)$. If $T_{j}(v_{i1})\neq T_{j}(v_{i2})$ then the distance between $v_{i1}$ and $v_{i2}$ belongs to $[2^j, 2^{j+1}]$, hence $\varphi_j(v_{i1})\neq \varphi_j(v_{i2})$ and $c(v_{i1})\neq c(v_{i2})$.\\
Now let us consider the number of colors. For a fixed $j$, all the vertices from $B_j$ has the same first term. Let us consider $v_i$ with the highest second term. Since $\varphi_j$ is a 12-coloring then $\varphi_j (v_i)+12\cdot t(v_i)\leq 12\cdot (t(v_i)+1)$. Since all vertices from $T(v_i)$ create a clique, and there are already $t(v_i)+1$ of them when $v_i$ is being colored, we obtain $t(v_i)+1\leq \omega(G)$. Hence the second term of $c(v_i)$ is less or equal to $12\omega(G)$.
Since there are at most $\lceil \log_2(\sigma)\rceil$ values of $j$ (see explanation above lemma \ref{lemBranchFF}) and for each $j$ we have at most $12\omega(G)$ second terms of $c$, we use no more than $\lceil \log_2(\sigma)\rceil\cdot 12\omega(G)$.
\end{proof}

\begin{cor} For any  sequence of $\sigma$-disks $(D_i)_{i\in [n]}$ and $(\varphi_i)$ as on Theorem \ref{thBranchColor}, we have  $cr(BranchColor_{(\varphi_j)}((D_i)_{i\in [n]}))\le12\lceil\log_2(\sigma)\rceil$.
\end{cor}

For the three algorithms we presented so far, clearly the $BranchColor$ has the best competitive ratio for 'large' $\sigma$. Since the competitive ratio of $SimpleColor$ depends on the choice of coloring of the plane, which in turn depends on $\sigma$, it is a bit hard to pinpoint the exact value of $\sigma$ where $BranchColor$ becomes better than $SimpleColor$. However $cr(SimpleColor)$ is quadratic in terms of $\sigma$, no matter the choice of coloring of $G_{[1,\sigma]}$. Figure \ref{wykresBasicRatios} shows the competitive ratios of the algorithms, where we take coloring of $G_{[1,\sigma]}$ from Proposition \ref{prop:hkwadrat} for $SimpleColor$.

\begin{figure}[h]\begin{center}
\includegraphics[width=1\textwidth]{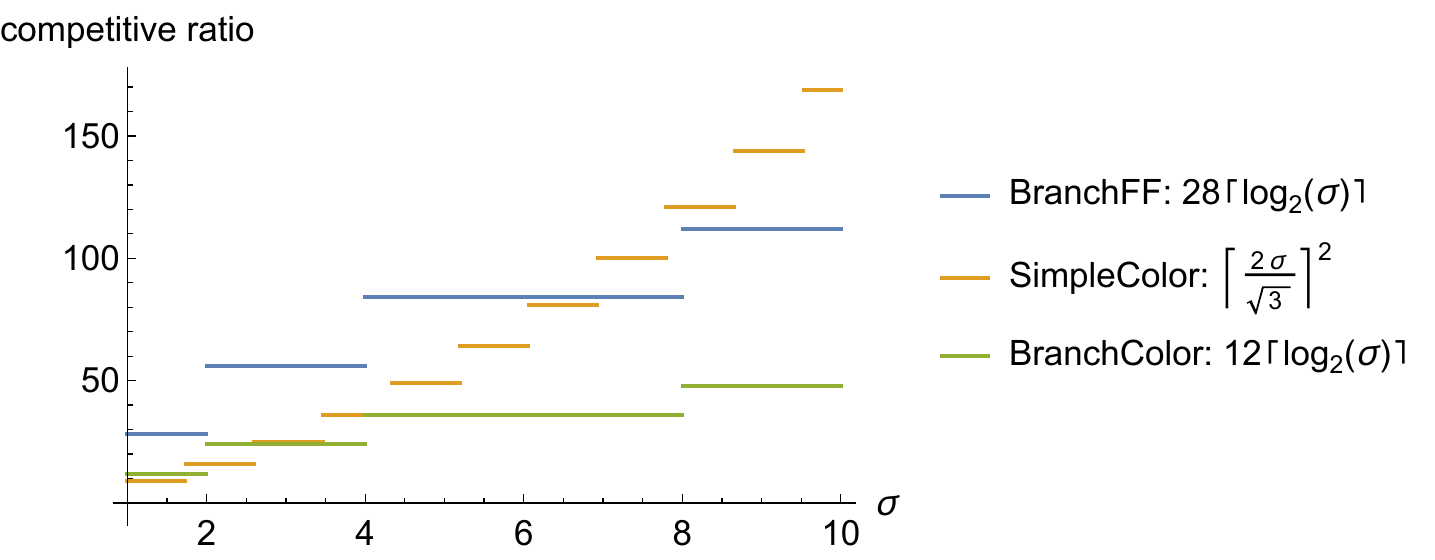}
\caption{A comparison of competitive ratios of algorithms $BranchFF$, $SimpleColor$ (with colorings of the plane from Proposition \ref{prop:hkwadrat}), $BranchColor$.}
\label{wykresBasicRatios}
\end{center}\end{figure}

\section{Algorithms with \emph{folding}}

The next algorithm $FoldColor$ follows the same idea as $SimpleColor$, with a difference that it uses $b$-fold coloring instead of the classical coloring of the plane. It was introduced by Junosza-Szaniawski, Sokół, Rzążewski, and Węsek \cite{udg} as a coloring algorithm for unit disk graphs, however, it can also be applied for the more general case of $\sigma$-disk graphs, with different colorings of the plane. 
The algorithm goes as follows. We start with some fixed solid $b$-fold coloring $\varphi=(\varphi_1,\ldots,\varphi_b)$  of $G_{[1,\sigma]}$ with colors $[k$]. When a disk $D_i$ is read, it is assigned to a single tile from the coloring in two steps. First, we find all tiles that contain the center $v_i$ of $D_i$. Then we choose one of the $b$ layers of $\varphi$ for $D_i$ and call it $\ell(v_i)$ (we try to distribute disks to layers as uniformly as possible). By doing so we choose a single tile $T$ from that layer containing $v_i$. We count the number of previous vertices assigned to this tile - $t(v_i)$. Then $v_i$ is colored with the color of this tile plus $k$ times the number of previous vertices of $T$. 

Notice that for $b=1$ the algorithm $Color_\varphi$ is the same as $SimpleColor_\varphi$.

\begin{algorithm}[H]\label{Alg:FoldColor}
\caption {$FoldColor_\varphi((D_i)_{i\in [n]}$)}
\ForEach {$i\in [n]$}{Read $D_i$, let $v_i$ be the center of $D_i$\\
\ForEach {$r\in [b]$}{ let $T_r(v_i)$ be the tile from the layer $r$ containing $v_i$}
$\ell(v_i)\gets1+(|\{v_1,\ldots v_{i-1}\}\cap \bigcap _{r\in [b]}T_r(v_i)|)_b$ \label{layer}\\
$t(v_i) \gets |\{u\in \{v_1,\ldots v_{i-1}\}\cap T_{\ell(v_i)}(v_i): \ell(u)=\ell(v_i) \}|$\\
$c(v_i)\gets \varphi _{\ell(v_i)}(v_i)+k\cdot t(v_i)$\label{linej:c}
}
\Return {$c$} \label{lineNo_b}
\end{algorithm}

The following Theorem can be found in the paper by Junosza-Szaniawski et al. \cite{udg}, where only colorings of Unit Disk Graphs are considered. Hence we adjust the proof for our case of $\sigma$-DGs. The bound on the number of colors remains unchanged, as it relies directly on $k$ rather than on $\sigma$ (but naturally $k$ grows along with $\sigma$).

\begin{theorem}\label{poprawnoscTempFold}
Let $\varphi=(\varphi_1,\ldots,\varphi_b)$ be a solid $b$-fold $k$-coloring of $G_{[1,\sigma]}$, and $(D_i)_{i\in [n]}$ sequence of $\sigma$-disks. Algorithm $FoldColor_\varphi((D_i)_{i\in [n]}$) returns  coloring of $G=G((D_i)_{i\in [n]})$.
\end{theorem}

\begin{proof}
Consider any two centers of disks $v_i$ and $v_j$ at euclidean distance at most $\sigma$. If $T_{\ell(v_i)}(v_i)=T_{\ell(v_j)}(v_j)$ then $\varphi_{\ell(v_i)}(v_i)=\varphi_{\ell(v_j)}(v_j)$ but 
$t(v_i)\neq t(v_j)$ and by formula in line \ref{linej:c} we get $c(v_i)\neq (c_j)$. If $T_{\ell(v_i)}(v_i)\neq T_{\ell(v_j)}(v_j)$ then either $\ell(v_i)\neq \ell(v_j)$ or the distance between $v_i$ and $v_j$ is within $[1,\sigma]$. Hence $\varphi_{\ell(v_i)}(v_i)\neq \varphi_{\ell(v_j)}(v_j)$ and $c(v_i)\neq c(v_j)$.
\end{proof}

\begin{theorem}[Junosza-Szaniawski et al. \cite{udg}]\label{warstwycol}
Let $\varphi=(\varphi_1,\ldots,\varphi_b)$ be a solid $b$-fold $k$-coloring of $G_{[1,\sigma]}$, and $(D_i)_{i\in [n]}$ sequence of $\sigma$-disks. Algorithm $FoldColor_\varphi((D_i)_{i\in [n]})$ returns  coloring of $G=G((D_i)_{i\in [n]})$ with the largest color at most $k\cdot\left \lfloor \frac{\omega(G)+(b-1)\gamma}{b}\right \rfloor$, where $\gamma$ is the maximum number of subtiles in one tile of  $\varphi$.
\end{theorem}

Proof of a slightly improved result will be included below, hence we omit this one.

Notice that, by Lemma \ref{gamma}, $k\cdot\left \lfloor \frac{\omega(G)+(b-1)\gamma}{b}\right \rfloor\leq k\cdot\left \lfloor \frac{\omega(G)}{b}+6(b-1)\right \rfloor\leq \frac{k}{b}\cdot \omega(G) +6k (b-1)$. It shows that the algorithms competitive ratio is close to $\frac{k}{b}$ for graphs with large $\omega(G)$. But we are at a disadvantage if the graph has a small clique size, especially when $\sigma$ is big, as $\frac{k}{b}$ is quadratic in terms of $\sigma$.
In the case of UDGs with relatively large clique size ($\omega(G) \geq 108901$ for a $25$-fold coloring from \cite{udg}) the algorithm, $FoldColor$ has a competitive ratio lower than $5$, which was the previous best, obtained by $FirstFit$. The fact that $FirstFit$ can be beaten by $FoldColor$ is especially interesting for us since later in this paper we replace the $FirstFit$ part of $BranchFF$ with $FoldColor$.

Now we will slightly improve the $FoldColor$ algorithm by better distributing the vertices among layers. This change will be carried on to the next algorithm as well.
Notice that in line \ref{layer} of the algorithm $FoldColor_\varphi((D_i)_{i\in [n]})$ the first vertex in each subtile is assigned to the first layer. Hence in each tile the numbers of vertices from the first and last layers can differ by $\gamma$. To lessen this difference we can distribute vertices among layers in a way closer to uniform. To do so we need the following definition.

\begin{defi} Let $\varphi$ be a solid coloring of $G_{[1,\sigma]}$ with $\gamma$ subtiles in each tile. A function $\eta$ assigning a number from $[b]$ to each subtile is called \emph{shading} of $\varphi$ if in every tile the numbers of subtiles with the same value of $\eta$ are the same. 
\end{defi}

\begin{figure}[h]
\center
\includegraphics[scale=.8]{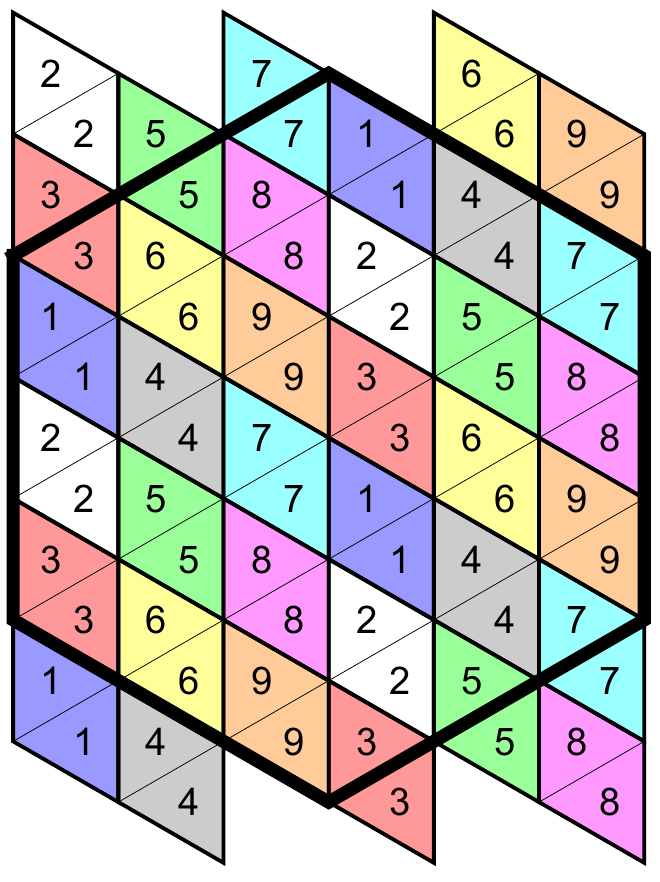}\caption{Example of a shading for tiling with $h=3$}\label{shade}
\end{figure}
Notice that the shading does not depend on the coloring $\varphi$ but rather on the tiling that it is based on. For the hexagonal tilings defined in section \ref{sec:tilings}, such shading exists (see an example on Figure \ref{shade}). The construction for $h\ge3$ is based on a fact that the vertical borders of the hexagons create line of the form $x=s\cdot \frac{\sqrt{3}}{4h}$, where $s\in \mathbb{Z}$ (it follows from the proof of Lemma \ref{gamma} presented in \cite{udg}). It gives a partition of the plane into vertical stripes - each of these is assigned $h$ shades in a cyclic manner: choose any stripe to be the \emph{first} one - it receives shades $\{1,2,...,h\}$, the one next to it receives shades $\{h+1,...,2h\}$ and so on. Each stripe is divided into diamonds (again by the borders of the hexagons), consisting of two subtiles. We assign shades to the diamonds in a cyclic manner. It is easy to align the stripes in such a way that every hexagon has exactly 6 triangular subtiles of each shade.

For given solid coloring $\varphi$ of $G_{[1,\sigma]}$ and its shading $\eta$ we can define an improved algorithm $FoldShadeColor_{\varphi,\eta}((D_i)_{i\in [n]})$. It is an algorithm obtained from $FoldColor_\varphi((D_i)_{i\in [n]})$ by replacing line \ref{layer} with \\ $\ell(v_i)\gets \eta(\bigcap _{r\in [b]}T_r(v_i)) +(|\{v_1,\ldots v_{i-1}\}\cap \bigcap _{r\in [b]}T_r(v_i)|)_b$ (so the very first vertex from a subtile with shade equal $t$ is assigned to the $t$-th layer). 

\begin{theorem}\label{th:FoldShade}
Let $\varphi=(\varphi_1,\ldots,\varphi_b)$ be a solid $b$-fold $k$-coloring of $G_{[1,\sigma]}$, $\eta$-shading, and $(D_i)_{i\in [n]}$ sequence of disks. Algorithm $FoldShadeColor_{\varphi,\eta}((D_i)_{i\in [n]}$) returns  coloring of $G=G((D_i)_{i\in [n]})$ with the largest color at most $k\cdot\left \lfloor \frac{\omega(G)+(b-1)\gamma/2}{b}\right \rfloor$, where $\gamma$ is the maximum number of subtiles in one tile of  $\varphi$.
\end{theorem}
\begin{proof}

The proof is analogous to the proof of Theorem \ref{warstwycol}\cite{udg} with a difference in the estimation of $\omega(G)$.

Let $v_i$ be a vertex that got the biggest color. Consider the moment of the course of the algorithm when vertex $v_i$ was colored. Let $\ell_i=\ell(v_i),$ and $t(v_i), c(v_i)$ be defined as in the algorithm. Let $T=T_{\ell_i}$ be the tile from the $\ell_i$-th layer containing $v_i$. Let $S_1,\ldots S_\gamma$ be subtiles of $T$. 
Let $s_q=|\{u\in \{v_1,\ldots v_i\}: u\in S_q\} |$
 and  $s_q(\ell_i)=|\{u\in \{v_1,\ldots v_i\}: u\in S_q, \ell(u)=\ell_i\} |$ for $q\in[\gamma]$. So the number of all vertices assigned to layer $\ell_i$ and tile $T$ before $v_i$ equals $t(v_i)=\sum_{q=1}^\gamma s_q (\ell_i)-1$.

 Now let us look closely at the values of $s_q$ in terms of $s_q(\ell_i)$. We denote $(x)_b:=x \mod b$. Vertices in $S_q$ are assigned to layers in a cyclic manner starting from $\eta(S_q)$ and there are $(\ell_i-\eta(S_q))_b$ of them before the first one is assigned to layer $\ell_i$ (if such vertex exists). Then we know that there are $s_q(\ell_i)$ vertices in $S_q$ assigned to $\ell_i$, so there are at least $b(s_q(\ell_i)-1)+1$ vertices across all layers. Hence $s_q\ge b(s_q(\ell_i)-1)+(\ell_i-\eta(S_q))_b+1$ (if $s_q(\ell_i)=0$ the right side of the inequality is negative while the left side is non-negative, so it holds). 
  
 Now we are ready to estimate the number of all vertices from $\{v_1,\ldots,v_i\}$ contained in $T$ (but not necesarily assigned to the layer $\ell_i$). Notice that these vertices are pairwise at distance less than one and hence they form a clique. After summing both sides of the last inequality over $q$ from 1 to $\gamma$ we obtain $$\omega(G) \ge\sum_{q=1}^\gamma s_q \ge
\sum_{q=1}^\gamma [b(s_q(\ell_i)-1)+(\ell_i-\eta(S_q))_b+1]=$$

$$b\sum_{q=1}^\gamma s_q(\ell_i)-b\gamma +\sum_{q=1}^\gamma (\ell_i-\eta(S_q))_b+\gamma
\overset{(*)}=$$

$$\overset{(*)}=
 b(t(v_i)+1)-(b-1)\gamma+\frac{\gamma}{b}\sum_{j=0}^{b-1} j=$$

$$=b(t(v_i)+1)-(b-1)\gamma+\frac{\gamma}{b}\frac{b(b-1)}{2}=
 b(t(v_i)+1)-(b-1)\gamma/2$$ 
 
Notice that since $\eta$ is a shading of $\varphi$ the number of subtiles in the tile $T$ with each shade from 1 to $b$ is the same. Thus $(\ell_i-\eta(S_q))_b$ admits each of the values from 0 to $b-1$ for the same number of subtiles, namely $\gamma/b$. This explains equality $(*)$. 
 
 From the estimation above we obtain that $t(v_i)+1$ is at most $\left \lfloor \frac{\omega(G)+(b-1)\gamma/2}{b}\right \rfloor$. Since we chose $v_i$ to have the biggest color and the total number of colors equals at most $c(v_i)= k (t(v_i)+1)$.
\end{proof}

In our algorithms it is crucial to construct good $b$-fold colorings of $G_{[1,\sigma]}$. We already showed a method for constructing such colorings. However the best possible $(h^2,p,q)$-coloring depends highly on the value of $\sigma$. Thus it is hard to give best attainable bounds on the competitive ratio without specifying $\sigma$. Hence we use Corollary \ref{prop:hkwadrat} to give some proper bounds on the competitive ratio of our algorithms, although for most values of $\sigma$ it is possible to lower them by adjusting the coloring of $(G_{[1,\sigma]})$. 

\begin{cor}
For the $h^2$-fold $\varphi$ coloring of the plane  from Theorem \ref{prop:hkwadrat} and any sequence of $\sigma$-disks $(D_i)_{i\in [n]}$ let $\omega=\omega(G((D_i)_{i\in [n]}))$. Then 
  $$\mathrm{cr}(FoldColor_\varphi((D_i)_{i\in [n]})) \leq   \frac{\left \lceil (\frac{2\sigma}{\sqrt{3}}+1)\cdot h\right \rceil^2}{\omega}\cdot\left \lfloor \frac{\omega+(h^2-1)6h^2}{h^2}   \right \rfloor$$ 

and  $$\mathrm{cr}(FoldShadeColor_{\varphi,\eta}((D_i)_{i\in [n]})) \leq   \frac{\left \lceil (\frac{2\sigma}{\sqrt{3}}+1)\cdot h\right \rceil^2}{\omega}\cdot\left \lfloor \frac{\omega+(h^2-1)3h^2}{h^2}   \right \rfloor,$$ 

Hence the ratios are of order $\left(\frac{2\sigma}{\sqrt{3}}+1\right)^2+O\left(\frac{1}{h}\right)+O\left(\frac{h^4}{\omega} \right).$
\end{cor}

\begin{figure}[h]
\center
\includegraphics[width=\textwidth]{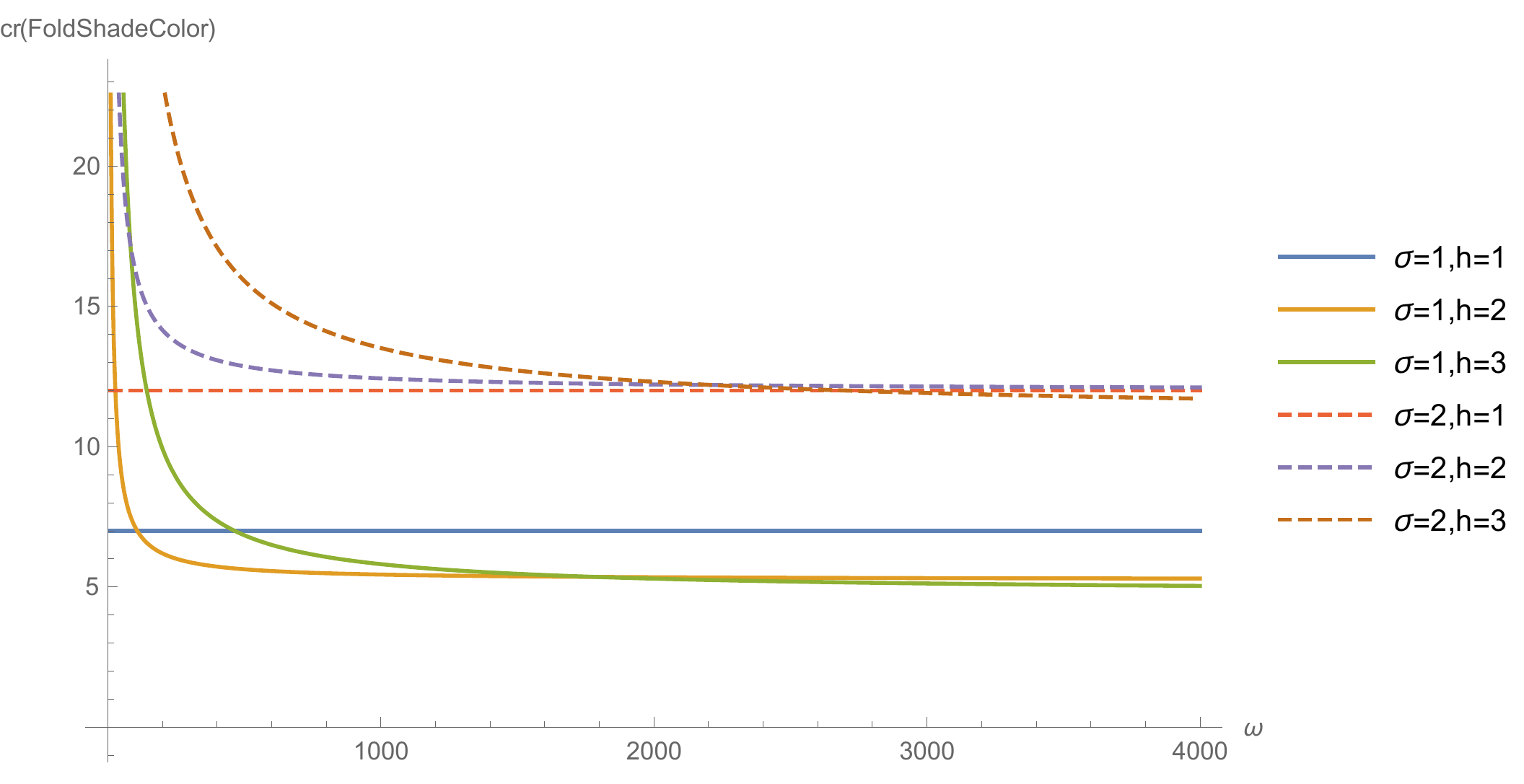}\caption{Competitive ratio of $FoldShadeColor$ for $\sigma$-DG, $\sigma=1$ or $\sigma=2$, depending on $\omega$, with $\varphi$ as $h^2$-fold colorings where $h=1,2$ or $3$.}\label{wykres:ratio}
\end{figure}
\begin{figure}[h]
\center
\includegraphics[scale=.8]{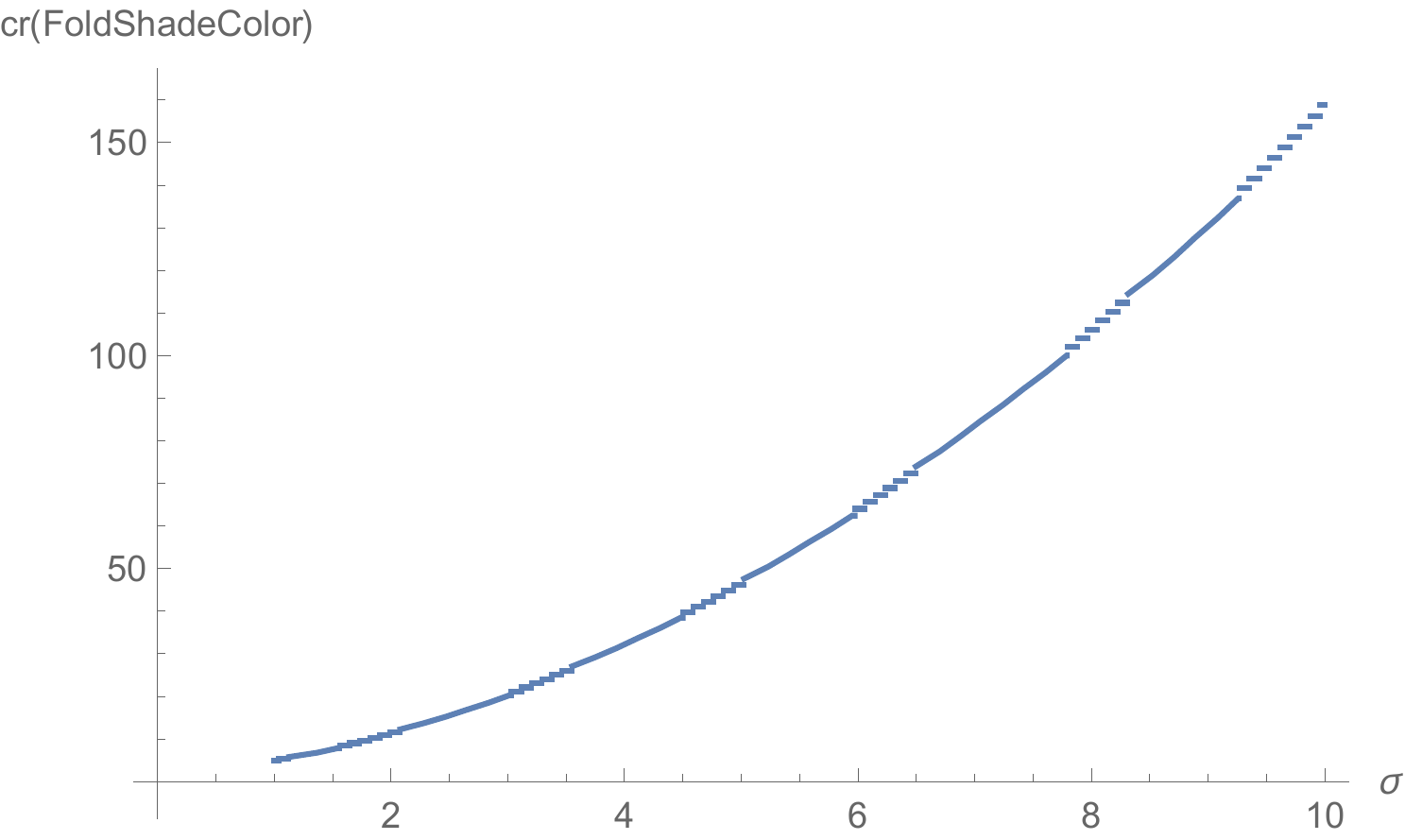}\caption{Competitive ratio of $FoldShadeColor$ for $\sigma$-DG with $\omega=10^9$ depending on $\sigma$, with $\varphi$ as $h^2$-fold colorings where $h=10$.}\label{wykres:ratio_sigma}
\end{figure}
So the algorithm $FoldShadeColor$ has a better competitive ratio than $FoldColor$, however competitive ratios of both algorithms are equal asymptotically. Notice that for $h=5$ and unit disk graphs $G$ with $\omega(G) \geq 108901$, the competitive ratio of the algorithm $FoldColor$ is less than $5$, hence lower than the ratio of the FirstFit. For the algorithm $FoldShadeColor$ with $h=5$ the ratio is below 5 for UDG with $\omega\ge 54450$ (and occasionally for smaller $\omega$, because of the floor function in our bound). See Figures \ref{wykres:ratio}, \ref{wykres:ratio_sigma} to study the changes in ratio depending on $\omega$, $\sigma$ and $h$.

The last algorithm $BranchFoldColor_{\varphi,\eta}$ joins all the techniques from the previous ones. The vertices are dived in two steps as in $BranchColor$ but rather than coloring of $B_j$ by using a regular coloring of the plane we use a $b$-fold coloring of $G_{[1,2]}$ and a shading $\eta$ much like in the $FoldShadeColor{\varphi,\eta}$ algorithm.

\begin{algorithm}[H]
\caption {$BranchFoldColor_{\varphi,\eta}((D_i)_{i\in [n]}$)}
\ForEach {$i\in [n]$}{Read $D_i$, let $v_i$ be the center of $D_i$\\
\eIf{$\sigma_i=\sigma=2^t,\ t\in\mathbb{Z}_{+}$}{$j(v_i)\gets\lfloor\log_2\sigma_i\rfloor-1$}
{$j(v_i)\gets\lfloor\log_2\sigma_i\rfloor$}
$B_j := B_j \cup \{v_i\}$\\
\ForEach {$r\in [b]$}{ let $T_r(v_i)$ be the tile from the layer $r$ containing $v_i$}
$\ell(v_i)\gets \eta(\bigcap _{r\in [b]}T_r(v_i)) +(|\{v_1,\ldots v_{i-1}\}\cap \bigcap _{r\in [b]}T_r(v_i)|)_b$ \label{BFlayer}\\
$t(v_i) \gets |\{u\in \{v_1,\ldots v_{i-1}\}\cap T_{\ell(v_i)}(v_i): \ell(u)=\ell(v_i) \}|$\\
$c(v_i)\gets (j,\varphi _{\ell(v_i)}(v_i)+k\cdot t(v_i))$ \label{BFlinej:c}
}
\Return {$c$} \label{BFlineNo_b}
\end{algorithm}

\begin{theorem}\label{th:BranchFold}
Let $\varphi=(\varphi_1,\ldots,\varphi_b)$ be a solid $b$-fold $k$-coloring of $G_{[1,2]}$, $\eta$-shading, and $(D_i)_{i\in [n]}$ sequence of disks. Algorithm $BranchFoldColor_{\varphi,\eta}((D_i)_{i\in [n]}$) returns  coloring of $G=G((D_i)_{i\in [n]})$ with the largest color at most $\lceil\log_2(\sigma)\rceil \cdot k\cdot\left \lfloor \frac{\omega(G)+(b-1)\gamma/2}{b}\right \rfloor$, where $\gamma$ is the maximum number of subtiles in one tile of  $\varphi$.
\end{theorem}
\begin{proof}
As in previous algorithms the branching partitions vertices into sets $B_j$ and colors each of these sets separately. The first element of the color of a vertex is the index of its set, hence there cannot be a conflict of colors between vertices from different branching sets. The second part of the color of a vertex comes from the same procedure as in algorithm $FoldShadeColor$, hence it gives a proper coloring of a graph induced by $B_j$ and the maximal value it takes is at most $k\cdot\left \lfloor \frac{\omega(G)+(b-1)\gamma/2}{b}\right \rfloor$. Since there are at most $\lceil\log_2(\sigma)\rceil$ sets $B_j$, we obtain the stated bound.
\end{proof}

The number of colors in this algorithm looks very similar to the one for $FoldShadeColor$, except for multiplying by $\lceil\log_2(\sigma)\rceil$. But it can be significantly smaller, since the values of $k$ and $b$ correspond to the coloring of $G_{[1,2]}$ instead of $G_{[1,\sigma]}$ (and with the same $b$ the value of $k$ is quadratic in terms of $\sigma$).

By using the fact that $\gamma\le 6b$ we obtain the following simplified bound.

\begin{cor}
The competitive ratio of $BranchFoldColor$ is at most:\\
$$\frac{\lceil\log_2(\sigma)\rceil\cdot k\lfloor\frac{\omega(G)}{b}+3(b-1)\rfloor}{\omega(G)}= O(\log_2(\sigma)\cdot \frac{k}{b})$$.
\end{cor}
Lets consider a function $f(b)$ that gives the  the best possible value of $k/b'$ for all $b'$-fold $k$-colorings of $G_{[1,2]}$ with $b'\le b$. Not only is it non-increasing, but also the value decreases from time to time. However, if we want a small competitive ratio of the algorithm, then clearly $\omega(G)$ should be large compared to $b$. For that reason, we should only consider small values of $b$. In the table below, we have chosen three \emph{good} $b$-fold colorings of $G_{[1,2]}$. The last column shows the minimal value of $\omega(G)$ where the upper bound on the number of colors is smaller than for the coloring from the previous row. Of course, the algorithm still works for smaller values of $\omega$ than the ones given below, but it is simply not the best option in such cases. 
\begin{table}[h]\caption{The first 3 columns correspond to the parameters of chosen colorings of $G_{[1,2]}$ (the best we get for $b=h^2\le100$). The last column shows the minimal values of $\omega(G)$, where the competitive ratio of $BranchFoldColor_{\varphi,\eta}$ is smaller then for the previous row.}
\begin{tabular}{c|c|c|c}
$b$ & $k$ & $\frac{k}{b}$& $\omega$\\\hline
1&12 &12& \\
9 & 100 &11.11 & 2700\\
64& 703 &10.98 & $1.02944\cdot 10^6$
\end{tabular}
\end{table}

\section{Coloring various shapes}

Many classes of intersection graphs of geometric shapes are considered in research papers. In particular, there is plenty of information available in the case of intersection graphs of axis-parallel rectangles. However, in this paper, we concentrate on the scaling aspect of shapes. In this section, we consider the online coloring of any convex shape. In fact, the methods we present could be applied to some non-convex shapes as well, but in general, it might not bring the best number of colors.

Let us first consider an intersection graph $G$ of similar copies of a convex shape $S$ - scaling by the factor less or equal $\sigma$, translation, rotation, and reflection are allowed. We assume that $S$ is the smallest of all similar copies.
One natural method of online coloring such graphs would be to first find circumscribed circles for all shapes in $G$. Then we "fill them in" - substitute shapes with disks that contain them, and color a resulting supergraph $H$ of $G$, which is a disk graph. All these operations can be done in an online setting. Unfortunately, the offline parameters of the supergraph $H$ may vary from those of the original graph $G$. In particular, the clique number may increase significantly. Hence if we were to calculate the ratio of diameters of $H$ and use one of our algorithms the number of colors will be expressed using $\omega(H)$, rather than $\omega(G)$, and we cannot tell what the competitive ratio of said algorithm would be. Fortunately, we can avoid such problems if every shape is replaced by a pair of disks.

To adapt our method for each shape $S$ we consider two disks with common center: $ID(S)$ - inner disk and $OD(S)$ - outer disk (described below, see Figure \ref{fig:shapedisk}). To be more precise we start with the given shape $S$ and choose a \emph{center of $S$} to be an interior point $P$. It is best if the choice of $P$ minimizes the ratio between the diameters of the outer and inner disks with centers in $P$. Just as we did with disks we will associate the centers of shapes with their corresponding vertices. The \emph{outer disk} is the smallest disk with the center in $P$ that contains shape $S$. The \emph{inner disk} is the largest disk with the center in $P$ contained in shape $S$. Let us call the diameters of $OD(S)$ and $ID(S)$ as outer and inner diameters, respectively, and denote their ratio as $\rho$. Without loss of generality, we can assume that the inner diameter of $S$ equals one and the outer $\rho$. Taking a similar copy of $S$ the translation, rotation and reflection do not effect the diameters. Since we allow scaling by a maximal factor of $\sigma$ and $S$ is the smallest of all copies, then all inner diameters are in $[1,\sigma]$ and all outer ones in $[\rho,\rho\cdot\sigma]$.
\begin{figure}[h]\begin{center}
\includegraphics[width=0.5\textwidth]{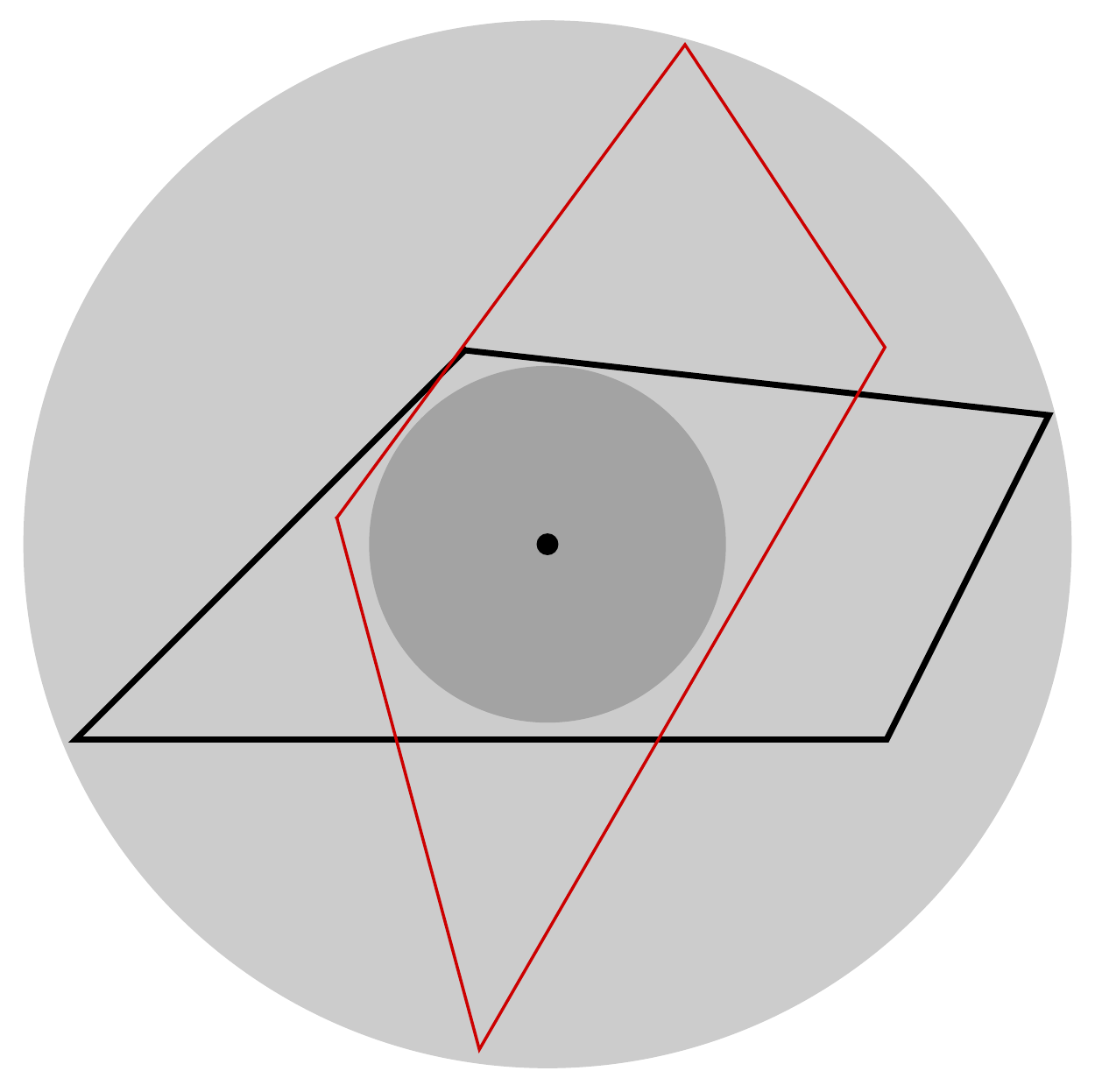}
\caption{Inner and outer disks of $S$. A rotated copy of $S$ has the same disks.}
\label{fig:shapedisk}
\end{center}\end{figure}

Notice that if two shapes intersect so do their outer disks. On the other hand, if two inner disks intersect so do their original shapes. Using these two simple facts we can find a coloring of the plane that can be used as a basis for an online coloring algorithm. If each tile of our coloring has diameter 1, as the inner diameter of $S$, then all shapes with centers in a single tile form a clique. Moreover, if tiles of the same color are at distance greater than $\rho\cdot\sigma$, then two shapes with centers in two same-colored tiles cannot intersect. Hence we are interested in $b$-fold coloring of $G_{[1,\rho\cdot\sigma]}$. With such coloring, we can use the $FoldShadeColor$ algorithm.

\begin{cor}\label{th:shapeRotate}
Let $(S_i)_{i\in [n]}$ sequence of shapes similar to $S$ with the ratio of outer and inner equal $\rho$.
Let $\varphi=(\varphi_1,\ldots,\varphi_b)$ be a solid $b$-fold $k$-coloring of $G_{[1,\rho\cdot\sigma]}$. Algorithm $FoldShadeColor_\varphi((S_i)_{i\in [n]})$ returns coloring of $G=G((S_i)_{i\in [n]})$ with the largest color at most $k\cdot\left \lfloor \frac{\omega(G)+(b-1)\gamma/2}{b}\right \rfloor$, where $\gamma$ is the number of subtiles in one tile of  $\varphi$.
\end{cor}
\begin{proof}
The proof of such coloring being proper is the same as in Theorem \ref{poprawnoscTempFold} (just change the wold disk to shape and $\sigma$ to $\rho\cdot\sigma$).

The number of colors is copied from Theorem \ref{th:FoldShade}. It follows from the estimation on $\omega(G)$, which is a result of the fact that all vertices from a single tile form a clique. It is still true for shapes, hence the value remains.
\end{proof}
 
 Notice that if we take $\sigma=1$, there is no scaling involved, hence there is no way of branching. Yet we still need a $b$-fold coloring of $G_{[1,\rho]}$. That is the reason why we include a general method of finding such colorings of the plane instead of considering only colorings of $G_{[1,2]}$. It also shows that the number of colors in our online coloring of $G$ depends highly on the value of $\rho$, which is why we try to minimize it with the choice of the center of $S$. For larger values of $\sigma$ however, the branching method is possible. Consider $\sigma_i$ to be the inner diameter of shape $S_i$, and use a $b$-fold coloring of $G_{[1,2\rho]}$. Then with branching on $\sigma_i$ we obtain the following.
\begin{cor}\label{th:shapeRotateBranch}
Let $(S_i)_{i\in [n]}$ sequence of shapes similar to $S$ with the ratio of outer and inner equal $\rho$.
Let $\varphi=(\varphi_1,\ldots,\varphi_b)$ be a solid $b$-fold $k$-coloring of $G_{[1,2\rho]}$. Algorithm $BranchFoldColor_\varphi((S_i)_{i\in [n]})$ returns coloring of $G=G((S_i)_{i\in [n]})$ with the largest color at most $\lceil log (\sigma)\rceil k\cdot\left \lfloor \frac{\omega(G)+(b-1)\gamma/2}{b}\right \rfloor$, where $\gamma$ is the number of subtiles in one tile of  $\varphi$.
\end{cor}

The last thing to point out is that our approach of online coloring similar shapes does not require shapes to be similar at all! In fact it is enough that all shapes from our sequence $(S_i)_{i\in [n]}$ have centers and disks $ID(S_i)$ and $OD(S_i)$ with diameters in $[1,\sigma]$ (or in any interval $[m,M]$ with $m,M\in\mathbb{R}_+$, $m<M$, since we can scale them all). Hence branching and coloring based on the coloring of a plane is a very general method.

\section{$L(2,1)$-labeling of disks}

In this section, we consider online $L(2,1)$-labeling of $\sigma$-disk graphs, which is a special case of coloring with integer values. Let us first recall the definition (quite often the labeling is defined starting with the value of 0, but it is more convenient for us to use positive integers only instead).
\begin{defi}
Formally, an $L(2,1)$-labeling  of a graph $G$ is any function $c:V\to \{1,\ldots,k\}$ such that
\begin{enumerate}
\item $|c(v) - c(w)| \geq 1$ for all $v,w \in V(G)$ such that $d(u,w)=2$ (we call such pairs \emph{second neighbors}),
\item $|c(v) - c(w)| \geq 2$ for all $v,w \in V(G)$ such that $vw \in E(G)$.
\end{enumerate}
The value $k-1$ is called a \emph{span} of the labeling.
\end{defi}
 The {\em $L(2,1)$-span} of a graph $G$, denoted by $\lambda(G)$, is the minimum span of an $L(2,1)$-labeling of $G$. The number of available labels is $\lambda(G)+1$, but some may be not used.

 $L(2,1)$-labeling was introduced by  Griggs and Yeh \cite{griggs} who proved that $\lambda(G)\le \Delta^2+ 2\Delta$ and conjectured $\lambda(G)\le \Delta^2$, which became known as the $\Delta^2$-conjecture. $L(2,1)$-labeling was consider for various classes of graphs  \cite{bodlaender, Cala, georges, heuvel}.
 
 Not all graphs can be effectively $L(2,1)$-labeled \emph{online}, since once you have two vertices of the same color, the next vertex could be the neighbor of them both, and thus the labeling would be faulty. However, having a geometric representation of a graph with some conditions can exclude those situations from happening. In those cases, we consider the competitive ratio as the ratio of the highest labels of online and offline labelings. Hence our maximal label is compared with $\lambda(G)+1$, and since $\lambda(G)+1\ge 2\omega(G)+1$, bounding the maximal label by a function of $\omega(G)$ is very convenient.
 Fiala, Fishkin and Fomin \cite{FFF} studied $L(2,1)$-labelings of disk graphs. One of their results states the following.
 \begin{lemma}\cite{FFF}
 There is no constant competitive online $L(2,1)$-labeling algorithm for the class of $\sigma$-disk graphs unless there is an upper bound on $\sigma$ and any $\sigma$-disk graph occurs as a sequence of disks in the online input.
 \end{lemma}
This result follows from the fact that in online $L(2,1)$-labeling once we color vertex $v$ with $c(v)$, the color $c(v)$ is not available for any current or future second neighbors. For that reason, we need to know whether a pair of vertices might become second neighbors in the future, and this is where the location of disks and the bounds on their diameters plays a role. If the centers of two disks are $C_1, C_2$ and the diameters $d_1$ and $d_2$, they can become second neighbors iff $\dist(C_1,C_2)\le \frac{d_1+d_2}{2}+\sigma$.
 
 Now we will consider which of the presented online coloring algorithms may be adjusted for $L(2,1)$-labelings. Firstly, we know that $SimpleColor$ can be used for online $L(2,1)$-labeling as it was presented in \cite{FFF}. The algorithm itself remains unchanged, but we need to use a special coloring of the plane, which we call solid $L(2,1)$-coloring of the plane.
 
 In case of unit disks it was proven in \cite{udg} that $FoldColor_\varphi$ gives proper $L(2,1)$-labeling as long as $\varphi$ is a solid $b$-fold $L^*(2,1)$-labeling of $G_{[1,1]}$ with $k$ labels. 

There is no reason for the algorithm to stop working properly for larger $\sigma$. Shading also does not spoil the labeling. Hence we get that 
  $FoldShadeColor_{\varphi,\eta}$ 'works' for $L(2,1)$-labeling $\sigma$-disk graphs. 
 
\begin{theorem}\label{th:warstwyL21}
Let $\varphi=(\varphi_1,\ldots,\varphi_b)$ be a solid $b$-fold $L^*(2,1)$-labeling of $G_{[1,\sigma]}$ with $k$ labels, $\eta$-shading, and $(D_i)_{i\in [n]}$ sequence of $\sigma$-disks. The algorithm $FoldShadeColor_\varphi((D_i)_{i\in [n]})$ returns an $L(2,1)$-coloring of $G=G((D_i)_{i\in [n]})$ with the largest label not exceeding $k\cdot\left \lfloor \frac{\omega(G)+(b-1)\gamma/2}{b}\right \rfloor$.
\end{theorem}

\begin{proof}
First, we prove the correctness of the algorithm. 
Consider any two centers of disks $v_i$ and $v_j$  with the same label. 
Notice that $T_{\ell(v_i)}(v_i)\neq T_{\ell(v_j)}(v_j)$, because otherwise $t(v_i)\neq t(v_j)$ and thus $c(v_i)\neq c(v_j)$. Since  $\varphi_{\ell(v_i)}(v_i)=\varphi_{\ell(v_j)}(v_j)$ and $\varphi$ is an $L^*(2,1)$-labeling of $G_{[1,\sigma]}$, $T_{\ell(v_i)}(v_i)$ and $T_{\ell(v_j)}(v_j)$ are at point-to-point distance greater than $2\sigma$. Hence $v_i$ and $v_j$ are at Euclidean distance greater than $2\sigma$ and thus are neither neighbors nor have a neighbor in common.
 
Now consider two  centers of disks $v_i$ and $v_j$ labeled with consecutive numbers. Without loss of generality assume that  $c(v_i)+1=c(v_j)$. Then $\varphi_{\ell(v_i)}(v_i)+1=\varphi_{\ell(v_j)}(v_j)$ or $\varphi_{\ell(v_i)}(v_i)=k$ and $\varphi_{\ell(v_j)}(v_j)=1$. Since $\varphi$ is an $L^*(2,1)$-labeling of the plane, $T_{\ell(v_i)}(v_i)$ and $T_{\ell(v_j)}(v_j)$ are at point-to-point distance greater than $\sigma$. Hence $v_i$ and $v_j$ are at Euclidean distance greater than $\sigma$ and are not adjacent.
 The bound on the largest color follows directly from Theorem \ref{th:FoldShade}. \end{proof} 
 
The use of $b$-fold labelings of $G_{[1,\sigma]}$ instead limiting ourselves to $1$-fold labelings might give significant improvement of the number of labels in labeling of $G=G((D_i)_{i\in[n]})$, just as it did in case of colorings. We have shown in section \ref{sec:planeL21} that the required solid labelings of $G_{[1,\sigma]}$ exist for any $\sigma$.

 Now let us consider the branching technique. While branching vertices we partition them according to their diameters into sets $B_j$ and color each set separately. This approach will not work with $L(2,1)$-labeling, since we cannot easily reserve a set of labels for each $B_j$ in advance. Since the labels of adjacent vertices must differ by at least 2, and vertices from sets $B_i,\ B_j$, $i\neq j$, can be neighbors, we would need to make sure that there is no pair of consecutive labels between $B_i$ and $B_j$. Adding to that, second neighbors must have different labels, and their common neighbor might be in a different set $B_j$ than either of them. Hence to the best of our knowledge, the $FoldShadeColor$ algorithm might be the best one for online $L(2,1)$-labeling $\sigma$-disk graphs. 
 
 Moreover, $FoldShadeColor$ can also be applied for online $L(2,1)$-labeling shapes with bounded inner and outer diameters. It is necessary to choose the appropriate coloring of the plane, as we did in the case of the online coloring of shapes, namely a solid $L^*(2,1)$-labeling of $G_{[1,\rho\cdot\sigma]}$.

 \section{Concluding remarks}
 
 We have considered a few algorithms which use some sort of coloring of the plane as a basis for online coloring geometric intersection graphs. As we have shown they can be applied not only for disk coloring but also the online coloring of series of shapes. Since there are many ways of modifying the coloring of the plane, there are many online coloring problems that could be approached with our methods. $L(2,1)$-labeling of $\sigma$-disks and shapes is just one of them. 
 
 Moreover, the same methods could be applied in other dimensions. One may consider a solid $b$-fold colorings of a line as a basis for online coloring, as in \cite{shorty}. The same could be done in higher dimensions, where for instance a solid coloring of $\mathbb{R}^3$ could serve for coloring intersection graphs of 3-dimensional balls and other three-dimensional solid objects. We must however always have bounds on the size of considered objects.

\end{document}